\documentclass[10pt,reqno]{amsart}
\usepackage{amssymb,amscd}

\newtheorem{theorem}{Theorem}[section]
\newtheorem{proposition}[theorem]{Proposition} 
\newtheorem{corollary}[theorem]{Corollary}
\newtheorem{lemma}[theorem]{Lemma}

\newtheorem{remarks}[theorem]{Remarks}

\newtheorem{claim}[theorem]{Claim}

\numberwithin{equation}{section}

\begin{document}

\title[Free Product von Neumann algebras]{Factoriality, Type Classification and Fullness \\ for Free Product von Neumann Algebras
}
\author[Y. Ueda]
{Yoshimichi UEDA}
\address{
Graduate School of Mathematics, 
Kyushu University, 
Fukuoka, 819-0395, Japan
}
\email{ueda@math.kyushu-u.ac.jp}
\thanks{Supported by Grant-in-Aid for Scientific Research (C) 20540213.}
\dedicatory{Dedicated to Professor Hideki Kosaki on the occasion of his 60th birthday}

\maketitle

\begin{abstract} 
We give a complete answer to the questions of factoriality, type classification and fullness for arbitrary free product von Neumann algebras.  
\end{abstract}

\allowdisplaybreaks{

\section{Introduction} Let $M_1$ and $M_2$ be $\sigma$-finite von Neumann algebras equipped with faithful normal states $\varphi_1$ and $\varphi_2$, respectively. The von Neumann algebraic free product $(M,\varphi)$ of $(M_1,\varphi_1)$ and $(M_2,\varphi_2)$ has been seriously investigated so far by utilizing Voiculescu's free probability theory and recently by Popa's deformation/rigidity theory. Some primary questions on $M$ apparently are factoriality, Murray--von Neumann--Connes type classification, and when $M$ becomes a full factor in the sense of Connes \cite{Connes:JFA74} under the separability of predual $M_*$. Several partial answers to the questions were given in the mid 90s by Barnett \cite{Barnett:PAMS95}, Dykema \cite{Dykema:DukeMathJ93},\cite{Dykema:Crelle94},\cite{Dykema:FieldsInstituteComm97}, and after then by us \cite{Ueda:MathScand01}. In particular, Barnett \cite{Barnett:PAMS95} provided a handy criterion for making $M$ be a full factor and also showed $(M_\varphi)'\cap M = \mathbb{C}$ under a slightly weaker assumption than the criterion. At the same time Dykema \cite{Dykema:Crelle94} investigated the question of factoriality of $M$ and computed the T-set $T(M)$ under some hypothesis. After then Dykema \cite{Dykema:FieldsInstituteComm97} gave a serious investigation relying on free probability techniques to the questions at least when the given $M_1$ and $M_2$ are of type I with discrete center. Despite those efforts it seems, to the best of our knowledge, that the questions are not yet settled completely. The main purpose of this paper is to give a complete answer to the questions. 

The new idea that we use in the present work is to work with the $\varphi$-preserving conditional expectation from $M$ onto $M_1$ (or $M_2$) instead of the free product state $\varphi$. The idea comes from the formulation of von Neumann algebraic HNN extensions (\cite{Ueda:JFA05}). (Recall that von Neumann algebraic HNN extension `$M=N\star_D\theta$' is formulated by using `the' conditional expectation from $M$ onto $N$ rather than that onto the smallest $D$.) The conditional expectation $E_1 : M \rightarrow M_1$ still `almost' satisfies the so-called freeness condition, and thus, in many cases, a suitable choice of faithful normal state, say $\psi$, on $M_1$ allows to make the new functional $\psi\circ E_1$ `almost' satisfy the freeness condition and have the large centralizer inside $M_1$. Based on this observation we will give a convenient partial answer to the questions, that is, we will prove that $M$ always becomes a factor of type II$_1$ or III$_\lambda$ with $\lambda\neq0$ and satisfies $M'\cap M^\omega = \mathbb{C}$ if at least one of $M_i$'s is diffuse without any assumption on the $\varphi_i$'s. This will be done in \S3. Some of the results presented in \S3 may be regarded as improvements of our previous results \cite{Ueda:MathScand01}, and thus we borrow some minor arguments from there without giving their details. In the next \S4, with the aid of Dykema's ideas in \cite{Dykema:DukeMathJ93} (and also \cite{Dykema:FieldsInstituteComm97}) we will give a complete answer to the questions. Roughly speaking our result tells that Dykema's `factoriality and type classification results' in \cite{Dykema:DukeMathJ93},\cite{Dykema:FieldsInstituteComm97} still hold true for {\it any} choice of $(M_1,\varphi_1)$ and $(M_2,\varphi_2)$. The precise statements are as follows. The resulting $M$ is always of the form $M_d \oplus M_c$ (possibly with $M = M_c$ or other words $M_d = 0$), where $M_d$ is a multi-matrix algebra and $M_c$ a factor of type II$_1$ or III$_\lambda$ with $\lambda \neq 0$. As a simple consequence no type III$_0$ factor arises as any (direct summand of) free product von Neumann algebra. The multi-matrix part $M_d$ can be described explicitly by Dykema's algorithm (see \cite[p.42--44]{Dykema:FieldsInstituteComm97}) and the type of $M_c$ is determined by the T-set formula $T(M_c) = \{ t \in \mathbb{R}\,|\,\sigma_t^{\varphi_1} = \mathrm{Id} = \sigma_t^{\varphi_2}\}$, though our proof cannot conclude $((M_c)_{\varphi|_{M_c}})'\cap M_c = \mathbb{C}1$ in general. Moreover $M_c' \cap M_c^\omega = \mathbb{C}$ always holds, and hence the asymptotic centralizer $(M_c)_\omega$ is trivial. We would like to emphasize that our proof uses only one easy fact \cite[Proposition 5.1]{Dykema:FieldsInstituteComm97} from Dykema's paper \cite{Dykema:FieldsInstituteComm97} (see Lemma \ref{L-2.2} for its precise statement) and is essentially independent of any free probability technique except the structure result on two freely independent projections (see \cite[Example 3.6.7]{VDN} and \cite[Theorem 1.1]{Dykema:DukeMathJ93}). Therefore our proof is self-contained and more remarkably rather short compared to \cite{Dykema:FieldsInstituteComm97}, though it does not work for identifying $(M_c)_{\varphi|_{M_c}}$ with an (interpolated) free group factor even when both $M_1$ and $M_2$ are of type I with discrete center. 

One of the motivations of the present work is a result due to Chifan and Houdayer \cite{ChifanHoudayer:DukeMathJ10}. In fact, by utilizing Popa's deformation/rigidity techniques Chifan and Houdayer \cite{ChifanHoudayer:DukeMathJ10} proved, among others, that any non-amenable free product factor is prime. Hence a necessary and sufficient condition of factoriality and `non-amenability' has been desirable for arbitrary free product von Neumann algebras. More than giving such a condition the main result of the present paper enables us to see, by Chifan and Houdayer's theorem \cite[Theorem 5.2]{ChifanHoudayer:DukeMathJ10}, that the diffuse factor part $M_c$ is always prime. Moreover a very recent work due to Houdayer and Ricard \cite{HoudayerRicard:Preprint10} also allows us to see that the diffuse factor part $M_c$ always has no Cartan subalgebra when $M_1$ and $M_2$ are hyperfinite (or amenable). In the final \S5 we give some remarks and questions related to the main theorem. 

This paper is written in the following notation rule: $\Vert-\Vert_\infty$ denotes an operator or $C^*$-norm. The projections and the unitaries in a given von Neumann algebra $N$ are denoted by $N^p$ and $N^u$, respectively, and the central support projection (in $N$) of a given $p \in N^p$ is denoted by $c_N(p)$. For a given von Neumann algebra $N$, its center is denoted by $\mathcal{Z}(N)$ and the unit of its {\it non-unital} von Neumann subalgebra $L$ by $1_L$. The GNS (or standard) Hilbert space associated with a given von Neumann algebra $N$ and a faithful normal positive linear functional $\psi$ on $N$ is denoted by $L^2(N,\psi)$ with norm $\Vert-\Vert_\psi$ and inner product $(-|-)_\psi$. Also the canonical embedding of $N$ into $L^2(N,\psi)$ is denoted by $\Lambda_\psi$. When no confusion is possible, we will often omit the symbol $\Lambda_\psi$ for simplicity; write $\Vert x\Vert_\psi$ instead of $\Vert\Lambda_\psi(x)\Vert_\psi$. The notations concerning the so-called modular theory entirely follow \cite{Takesaki:Book2} with the exception of $\Lambda_\psi$ (the symbol $\eta_\psi$ is used there instead). The other notations concerning free products and ultraproducts of von Neumann algebras will be summarized in the next \S2.  

\section{Notations and Preliminaries}
\subsection{Free products} For given $\sigma$-finite von Neumann algebras $M_1$ and $M_2$ equipped with faithful normal states $\varphi_1$ and $\varphi_2$, respectively, the (von Neumann algebraic) free product 
$$
(M,\varphi) = (M_1,\varphi_1)\star(M_2,\varphi_2)
$$
is defined to be a unique pair of von Neumann algebra $M$ with two embeddings $M_i \hookrightarrow M$ ($i=1,2$) and  faithful normal state $\varphi$ satisfying (i) $M = M_1\vee M_2$, (ii) $\varphi|_{M_i} = \varphi_i$ ($i=1,2$) and (iii) $M_1$ and $M_2$ are free in $(M,\varphi)$, i.e., $\varphi(x_1^\circ \cdots x_\ell^\circ) = 0$ whenever $x_j^\circ \in \mathrm{Ker}(\varphi_{i_j})$ with $i_j \neq i_{j+1}$, $j=1,\dots,\ell-1$. See \cite{VDN} for further details such as its concrete construction. 

Let $N$ be a von Neumann algebra with a faithful normal state $\psi$. For a subset $\mathcal{X}$ of $N$ we write $\mathcal{X}^\circ := \mathrm{Ker}(\psi) \cap \mathcal{X} = \{ x \in \mathcal{X}\,|\,\psi(x)=0\}$. Also, for given subsets $\mathcal{X}_1,\dots,\mathcal{X}_n$ of $N$, the set of all traveling words in subsets $\mathcal{X}_1,\dots,\mathcal{X}_n$ (inside $N$) is denoted by $\Lambda^\circ(\mathcal{X}_1,\dots,\mathcal{X}_n)$, where $x_1\cdots x_\ell$ is said to be a traveling word in the $\mathcal{X}_i$'s if $x_j \in \mathcal{X}_{i_j}$ with $i_j \neq i_{j+1}$, $j=1,\dots,\ell-1$. With these notations the above condition (iii) is simply re-written as $\varphi|_{\Lambda^{\circ}(M_1^\circ,M_2^\circ)} \equiv 0$. 

As shown in \cite[Lemma 1]{Barnett:PAMS95},\cite[Theorem 1]{Dykema:Crelle94} the modular automorphism $\sigma_t^\varphi$, $t \in \mathbb{R}$, is computed as `$\sigma_t^\varphi = \sigma_t^{\varphi_1}\star\sigma_t^{\varphi_2}$' which means that $\sigma_t^\varphi|_{M_i} = \sigma_t^{\varphi_i}$ ($i=1,2$). This fact immediately implies the next fact, which is a key of the present paper. It is probably well-known, but a proof is given for the reader's convenience. 

\begin{lemma}\label{L-2.1} There is a unique faithful normal conditional expectation $E_i : M \rightarrow M_i$ {\rm (}$i=1,2${\rm )} with $\varphi\circ E_i = \varphi$, and it satisfies $E_i|_{\Lambda^\circ(M_1^\circ,M_2^\circ)\setminus M_i^\circ} \equiv 0$. 
\end{lemma}
\begin{proof} The existence of $E_i$ follows from the above-mentioned formula of $\sigma_t^\varphi$ and Takesaki's theorem \cite[Theorem IX.4.2]{Takesaki:Book2}. We claim that $\Lambda_\varphi(M_i)$ and $\Lambda_\varphi\big(\mathrm{span}(\Lambda^\circ(M_1^\circ,M_2^\circ)\setminus M_i^\circ)\big)$ are orthogonal in $L^2(M,\varphi)$, and also that the map $\Lambda_\varphi(x) \mapsto \Lambda_\varphi(E_i(x))$, $x \in M$, is extended to the projection $\bar{E}_i$ from $L^2(M,\varphi)$ onto the closure of $\Lambda_\varphi(M_i)$. The first claim follows from the fact that $M_1$ and $M_2$ are free in $(M,\varphi)$, and the latter from the construction of $E_i$, see the proof of \cite[Theorem IX.4.2]{Takesaki:Book2}. Hence, for any $x \in \Lambda^\circ(M_1^\circ,M_2^\circ)\setminus M_i^\circ$ one has $\Lambda_\varphi(E_i(x)) = \bar{E}_i\Lambda_\varphi(x) = 0$ so that $E_i(x)=0$. \end{proof}  

In \S4 we will repeatedly use the next `free etymology' fact due to Dykema. We give a sketch of its proof for the reader's convenience.  

\begin{lemma}\label{L-2.2} {\rm(\cite[Proposition 5.1]{Dykema:FieldsInstituteComm97}; also see \cite[Theorem 1.2]{Dykema:DukeMathJ93})} 
Let $p\in\mathcal{Z}(M_1)^p$ be non-trivial and set $N := (\mathbb{C}p + M_1(1-p))\vee M_2$. Then $M_1 p$ and $pNp$ generate the whole $pMp$ and are free in $(pMp,(1/\varphi_1(p))\varphi|_{pMp})$. Moreover $c_M(p) = c_N(p)$. 
\end{lemma}   
\begin{proof} (Sketch) One can easily confirm that $M_1 p$ and $pNp$ generate $pMp$. In fact, for example, $p x_1 y_1 x_2 y_2 p = (x_1 p)(py_1 p)(x_2 p)(p y_2 p) + (x_1 p) (p y_1 x_2(1-p) y_2 p) \in M_1 p \cdot pNp \cdot M_1 p \cdot pNp + M_1 p \cdot pNp$ for $x_1,x_2 \in M_1$ and $y_1, y_2 \in M_2$. The freeness between $M_1 p$ and $pNp$ follows from the following fact: Note that $(\mathbb{C}p+M_1(1-p))+\mathrm{span}\big(\Lambda^\circ((\mathbb{C}p+M_1(1-p))^\circ,M_2^\circ)\setminus(\mathbb{C}p+M_1(1-p))^\circ\big)$ forms a dense $*$-subalgebra of $N$ in any von Neumann algebra topology, and thus any element in the kernel of $\varphi|_{pNp}$ can be approximated by a bounded net consisting of linear combinations of elements of the form $pxp$ with $x \in \Lambda^\circ((\mathbb{C}p+M_1(1-p))^\circ,M_2^\circ)$ whose leftmost and rightmost words are from $M_2^\circ$. It remains to see $c_M(p) = c_N(p)$. Clearly $c_M(p) \geq c_N(p)$ since $M \supseteq N$, and it suffices to see $c_M(p) \leq c_N(p)$. Since $c_N(p)$ commutes with $M_1(1-p)$ and moreover since $c_N(p) \geq p$, $c_N(p)$ must commute with $M_1 = M_1 p + M_1(1-p)$. Also $c_N(p)$ commutes with $M_2$ too so that $c_N(p) \in \mathcal{Z}(M)$. Hence $c_M(p) \leq c_N(p)$.  
\end{proof}

\subsection{Ultraproducts} (See \cite[Ch.5]{Ocneanu:LNM1138},\cite[\S II,\S III]{Connes:JFA74} and \cite[\S\S2.2]{Ueda:TAMS03} for details.) Let $N$ be a $\sigma$-finite von Neumann algebra. Note that the constructions in \cite[\S5.1]{Ocneanu:LNM1138} reviewed below are applicable even for $\sigma$-finite (= countably decomposable) von Neumann algebras as in \cite[\S II]{Connes:JFA74}. Take a free ultrafilter $\omega \in \beta(\mathbb{N})\setminus\mathbb{N}$. ($\beta(\mathbb{N})$ denotes the Stone--Cech compactification of $\mathbb{N}$.) Define $I_\omega(N)$ to be the set of all  $x=(x(m))_m \in \ell^\infty(\mathbb{N},N)$ such that $\lim_{m\rightarrow\omega}x(m)=0$ in strong$^*$ topology. The ultraproduct $N^\omega$ is defined as the quotient $C^*$-algebra of the multiplier $\mathcal{M}(I_\omega(N)) := \{ x \in \ell^\infty(\mathbb{N},N)\,|\, xI_\omega(N) \subseteq I_\omega(N), I_\omega(N)x \subseteq I_\omega(N)\}$ by its $C^*$-ideal $I_\omega(N)$, which becomes a von Neumann algebra. By sending $x \in N$ to the constant sequence $(x,x,\dots) \in \mathcal{M}(I_\omega(N))$ the original $N$ is embedded into $N^\omega$ as a von Neumann subalgebra. Any (faithful) normal positive linear functional $\psi$ on $N$ induces a (resp.~faithful) normal positive linear functional $\psi^\omega$ on $N^\omega$ defined by $\psi^\omega(x) = \lim_{m\rightarrow\omega}\psi(x(m))$ for $x \in N^\omega$ with representative $(x(m))_m \in \mathcal{M}(I_\omega(N))$. If $L$ is a von Neumann subalgebra with a faithful normal conditional expectation $E : N \rightarrow L$, then the natural embedding $\ell^\infty(\mathbb{N},L) \hookrightarrow \ell^\infty(\mathbb{N},N)$ gives a von Neumann algebra embedding $L^\omega \hookrightarrow N^\omega$ and $E$ can be lifted up to a faithful normal conditional expectation $E^\omega : N^\omega \rightarrow L^\omega$ that is induced from the mapping $(x(m))_m \in \mathcal{M}(I_\omega(N)) \mapsto (E(x(m)))_m \in \mathcal{M}(I_\omega(L))$.  

Define a smaller $C^*$-subalgebra $C_\omega(N)$ than $\mathcal{M}(I_\omega(N))$ to be the set of all $x = (x(m))_m \in \ell^\infty(\mathbb{N},N)$ with $\lim_{m\rightarrow\omega}\Vert [x(m),\chi]\Vert_{N_\star} = 0$ for every $\chi \in N_\star$, where $[x(m),\chi](y) := \chi(yx(m))-\chi(x(m)y)$ for $y \in N$. Then $C_\omega(N)$ still contains $I_\omega(N)$, and the asymptotic centralizer $N_\omega$ is defined to be the quotient $C^*$-algebra of $C_\omega(N)$ by $I_\omega(N)$ which becomes a von Neumann algebra. The inclusion relation $C_\omega(N) \subset \mathcal{M}(I_\omega(N))$ gives a von Neumann algebra embedding $N_\omega \hookrightarrow N^\omega$. It is not so hard to see that $N_\omega \subseteq N'\cap N^\omega$. However the equality does not hold in general. (In fact, such an explicit example exists. We learned it from Professor Masamichi Takesaki some years ago. We thank him for explaining it.) Thus $N'\cap N^\omega = \mathbb{C}$ implies that $N_\omega = \mathbb{C}$ . The reverse implication is known to hold only when $N$ is finite (see \cite[Corollary 3.8]{Connes:JFA74}), and not known in general. Remark that the property of `$N'\cap N^\omega = \mathbb{C}$' is a stably isomorphic one for factors. Namely, if $(pNp)'\cap (pNp)^\omega = \mathbb{C}$ for some non-zero $p \in N^p$, then so is the original $N$ when $N$ is known to be a factor. In fact, since $N$ is known to be a factor, $N\cong pNp$ or we can choose a smaller $e \in N^p$ than $p$ so that $N \cong (eNe)\bar{\otimes}B(\mathcal{K})$ for some Hilbert space $\mathcal{K}$. The first case is trivial. For the latter case, note first that $N \cong (eNe)\bar{\otimes}B(\mathcal{K})$ implies $(N\subset N^\omega) \cong (eNe \subset eN^\omega e)\bar{\otimes}B(\mathcal{K})$ since $N$ sits inside $N^\omega$, and hence $N'\cap N^\omega \cong (eNe)'\cap(eN^\omega e)$. Clearly $eN^\omega e = e(pNp)^\omega e$, and then $(eNe)'\cap(eN^\omega e) = \mathbb{C}e$ by \cite[Lemma 4.1]{Voeden:PLMS73} or \cite[Lemma 2.1]{Popa:InventMath81}. 

The ultraproduct $\mathcal{H}^\omega$ of $\mathcal{H} := L^2(N,\psi)$ is defined to be the quotient of $\ell^\infty(\mathbb{N},\mathcal{H})$ by the subspace consisting of all $(\xi(m))_m$ with $\lim_{m\rightarrow\omega}\Vert\xi(m)\Vert_{\mathcal{H}} = 0$. It becomes again a Hilbert space with inner product $(\xi|\eta)_{\mathcal{H}^\omega} = \lim_{m\rightarrow\omega}(\xi(m)|\eta(m))_\mathcal{H}$ for $\xi,\eta \in \mathcal{H}^\omega$ with representatives $(\xi(m))_m, (\eta(m))_m$, respectively. The GNS Hilbert space $L^2(N^\omega,\psi^\omega)$ can be embedded into $\mathcal{H}^\omega$ as a closed subspace by $\Lambda_{\psi^\omega}(x) \mapsto [(\Lambda_\psi(x(m)))_m]$ for $x \in N^\omega$ with representative $(x(m))_m$. 

\section{Analysis in The Diffuse Case} 

Throughout this section let $M_1$ and $M_2$ be $\sigma$-finite von Neumann algebras equipped with faithful normal states $\varphi_1$ and $\varphi_2$, respectively, and $(M,\varphi)$ be their free product, see \S\S2.1.  

The next is a slight generalization of \cite[Proposition 1]{Ueda:MathScand01}

\begin{proposition}\label{P-3.1} If $A$ is a diffuse von Neumann subalgebra of the centralizer $(M_1)_\psi$ with some faithful normal state $\psi$ on $M_1$ and if $x \in M$ satisfies $xAx^* \subseteq M_1$, then $x$ must be in $M_1$. 
\end{proposition}
\begin{proof} Let us make use of the idea in the proof of \cite[Lemma 2.5]{Popa:JOT83} as in \cite[Lemma 1]{Ueda:MathScand01}. The main difference from \cite[Lemma 1]{Ueda:MathScand01} is the use of $\psi\circ E_1$ instead of the free product state $\varphi$ itself. 

Note that any word in $\Lambda^\circ(M_1^\circ,M_2^\circ)\setminus M_1^\circ$ is of the form $ay^\circ b$ with a  word $y^\circ \in \Lambda^\circ(M_1^\circ,M_2^\circ)$ beginning and ending at $M_2^\circ$ and $a,b \in M_1^\circ\cup\{1\}$. In what follows let $ay^\circ b$ be such an alternating word. Then for any partition of unity $\{e_i\}_{i=1}^n$ in $A$ we have 
\begin{align*} 
|\psi\circ E_1(x^*(ay^\circ b))|^2 
&= 
\Big|\psi\circ E_1\Big(\sum_{i=1}^n e_i x^*ay^\circ b\Big)\Big|^2 \\
&= 
\Big|\psi\circ E_1\Big(\sum_{i=1}^n e_i x^*ay^\circ b e_i\Big)\Big|^2 \\
&\leq 
\Big\Vert \sum_{i=1}^n e_i x^* ay^\circ b e_i \Big\Vert_{\psi\circ E_1}^2 
\phantom{aaaaaa}  
\text{(by the Cauchy--Schwarz inequality)}\\
&= 
\sum_{i=1}^n \psi(e_i b^* E_1((y^\circ)^* (a^* x e_i x^* a) y^\circ) b e_i) \phantom{a\,}  
\text{(since $e_i e_j = 0$ if $i\neq j$)}\\
&= 
\sum_{i=1}^n \varphi_1(a^* x e_i x^* a)\,\psi(e_i b^* E_1((y^\circ)^* y^\circ) b e_i) \\
&\leq 
\varphi_1(a^* x x^* a)\,\Vert b^* E_1((y^\circ)^* y^\circ) b\Vert_\infty\,\max_{1\leq i\leq n}\psi(e_i), 
\end{align*} 
where the second line follows from the fact that $e_i \in A \subset (M_1)_\psi \subset M_{\psi\circ E_1}$ and the fifth line from Lemma \ref{L-2.1} and $(y^\circ)^* (a^* x e_i x^* a) y^\circ = \varphi_1(a^* x e_i x^* a)(y^\circ)^* y^\circ + (y^\circ)^* (a^* x e_i x^* a)^\circ y^\circ$ with $(a^* x e_i x^* a)^\circ := a^* x e_i x^* a - \varphi_1(a^* x e_i x^* a)1 \in M_1^\circ$ due to $xe_i x^* \in M_1$. Since $A$ is diffuse, $\max_{1\leq i\leq n}\psi(e_i)$ can be arbitrary small so that $x$ is orthogonal to $ay^\circ b$ in $L^2(M,\psi\circ E_1)$. 

Let $ay^\circ b$ be as above, and take arbitrary $c \in M_1$. By Lemma \ref{L-2.1} one has $\psi\circ E_1(c^*(ay^\circ b)) = \psi(c^* a E_1(y^\circ) b) = 0$, and thus $M_1$ and $\mathrm{span}(\Lambda^\circ(M_1^\circ,M_2^\circ)\setminus M_1^\circ)$ are orthogonal in $L^2(M,\psi\circ E_1)$. Clearly $M_1 + \mathrm{span}(\Lambda^\circ(M_1^\circ,M_2^\circ)\setminus M_1^\circ)$ is dense in $M$ in any von Neumann algebra topology, and hence we have $\Lambda_{\psi\circ E_1}(x) = \bar{E}_1\Lambda_{\psi\circ E_1}(x) = \Lambda_{\psi\circ E_1}(E_1(x))$ implying $x = E_1(x) \in M_1$, where $\bar{E}_1$ denotes the projection from $L^2(M,\psi\circ E_1)$ onto the closure of $\Lambda_{\psi\circ E_1}(M_1)$ induced from $E_1$, see the proof of Lemma \ref{L-2.1}. 
\end{proof}

The next corollary is immediate from the above proposition. 

\begin{corollary}\label{C-3.2} If $A$ is a diffuse von Neumann subalgebra of the centralizer $(M_1)_\psi$ for some faithful normal state $\psi$ on $M_1$, then $\mathcal{N}_M(A) = \mathcal{N}_{M_1}(A)$ and $A'\cap M = A'\cap M_1$, where $\mathcal{N}_P(Q)$ denotes the normalizer of $Q$ in $P$, i.e., the set of $u \in P^u$ with $uQu^*=Q$, for a given inclusion $P\supseteq Q$ of von Neumann algebras. 

In particular, any diffuse MASA {\rm(}semi-regular diffuse MASA, or singular diffuse MASA{\rm)} in $M_1$ which is the range of a faithful normal conditional expectation from $M_1$ becomes a MASA {\rm(}resp.~semi-regular MASA or singular MASA{\rm)} in $M$ too.    
\end{corollary}

As in \cite[Corollary 4]{Ueda:MathScand01} we can also derive the next corollary from the above proposition. 

\begin{corollary}\label{C-3.3} Suppose that there are a faithful normal state $\psi$ on $M_1$ and a diffuse von Neumann subalgebra $A$ of the centralizer $(M_1)_\psi$. Then, $A (= A\bar{\otimes}\mathbb{C}1) \subseteq M_1\rtimes_{\sigma^\psi}\mathbb{R} \subseteq M\rtimes_{\sigma^{\psi\circ E_1}}\mathbb{R}$ satisfies $A' \cap \big(M\rtimes_{\sigma^{\psi\circ E_1}}\mathbb{R}\big) = A' \cap \big(M_1\rtimes_{\sigma^\psi}\mathbb{R}\big)$. 
\end{corollary}

With the help of some general facts on von Neumann algebras we can derive the next factoriality and type classification result from the above proposition and corollaries.  

\begin{theorem}\label{T-3.4} If $M_1\neq\mathbb{C}\neq M_2$ and if either $M_1$ or $M_2$ is diffuse, then $M$ is a factor of type II$_1$ or III$_\lambda$ with $\lambda\neq0$, and its T-set is computed as $T(M) = \{ t \in \mathbb{R}\,|\,\sigma_t^{\varphi_1} = \mathrm{Id} = \sigma_t^{\varphi_2} \}$. 
\end{theorem}  

Remark that the type of $M$ is completely determined by the T-set in the above case, since $M$ is of type II$_1$ or III$_\lambda$ with $\lambda\neq0$. The proof below gives an explicit description of the (smooth) flow of weights of $M$, which explains that no type III$_0$ free product factor arises under our hypothesis. (Another explanation of this phenomenon is given later by showing $M_\omega = \mathbb{C}$, see Theorem \ref{T-3.7}, thanks to \cite[Theorem 2.12]{Connes:JFA74}.)

\begin{proof} Note that any diffuse $\sigma$-finite von Neumann algebra is (i) a $\sigma$-finite von Neumann algebra with diffuse center, (ii) a (possibly infinite) direct sum of non-type I factors, or (iii) a direct sum of algebras from (i) and (ii). Hence we may and do assume 
\begin{equation*} 
M_1 = Q_0 \oplus \sideset{}{^\oplus}\sum_{k\geq1} Q_k \quad \text{(a finite or an infinite direct sum)}, 
\end{equation*}
where $\mathcal{Z}(Q_0)$ is diffuse and the $Q_k$'s are non-type I factors. By \cite[Corollary 8]{ConnesStormer:JFA78} for the separable predual case, \cite[Theorem 11.1]{HaagerupStormer:AdvMath90} for the general $\sigma$-finite case, there is a faithful normal state $\psi_k$ on each $Q_k$ ($k\geq1$) so that $(Q_k)_{\psi_k}$ contains a diffuse von Neumann subalgebra $A_k$. Take a faithful normal state $\psi_0$ on $Q_0$, and define a faithful normal state 
\begin{equation*} 
\psi := (1/2)\psi_0\oplus\sideset{}{^\oplus}\sum_{k\geq1} (1/2^{k+1})\psi_k \quad \text{on} \quad M_1 = Q_0 \oplus \sideset{}{^\oplus}\sum_{k\geq1} Q_k. 
\end{equation*}
Then $(M_1)_\psi$ clearly contains the diffuse von Neumann subalgebra 
\begin{equation*} 
A := \mathcal{Z}(Q_0)\oplus\sideset{}{^\oplus}\sum_{k\geq1} A_k.
\end{equation*} 
Therefore, by Corollary \ref{C-3.2} and Corollary \ref{C-3.3} we have 
\begin{itemize} 
\item[(a)] $\big((M_1)_\psi\big)'\cap M = \big((M_1)_\psi\big)' \cap M_1$, 
\item[(b)] $(M_1\rtimes_{\sigma^\psi}\mathbb{R})'\cap(M\rtimes_{\sigma^{\psi\circ E_1}}\mathbb{R}) = \mathcal{Z}(M_1\rtimes_{\sigma^\psi}\mathbb{R})$.
\end{itemize}
By \cite[Corollary IX.4.22, Theorem X.1.7]{Takesaki:Book2} 
\begin{equation*} 
\Big(M\rtimes_{\sigma^{\psi\circ E_1}}\mathbb{R} \supseteq M_1\rtimes_{\sigma^\psi}\mathbb{R}\Big) \cong \Big(M\rtimes_{\sigma^{\varphi_1\circ E_1}}\mathbb{R} \supseteq M_1\rtimes_{\sigma^{\varphi_1}}\mathbb{R}\Big) 
\end{equation*} 
spatially, and thus (b) implies 
\begin{itemize} 
\item[(b')] $(M_1\rtimes_{\sigma^{\varphi_1}}\mathbb{R})'\cap(M\rtimes_{\sigma^\varphi}\mathbb{R}) = \mathcal{Z}(M_1\rtimes_{\sigma^{\varphi_1}}\mathbb{R})$ 
\end{itemize}
({\it n.b.}~$\varphi_1\circ E_1 = \varphi$). Hence, by the argument in \cite[Theorem 7]{Ueda:MathScand01} we get 
\begin{equation*}
\mathcal{Z}(M\rtimes_{\sigma^\varphi}\mathbb{R}) = (\mathbb{C}\rtimes\mathbb{R})\cap\mathcal{Z}(M_1\rtimes_{\sigma^{\varphi_1}}\mathbb{R})\cap\mathcal{Z}(M_2\rtimes_{\sigma^{\varphi_2}}\mathbb{R}) \subseteq \mathbb{C}\rtimes\mathbb{R}. 
\end{equation*}  
Hence we see that $M$ is a factor and not of type III$_0$ by ergodic theoretic argument as in \cite[Corollary 8]{Ueda:MathScand01} or by harmonic analysis argument, c.f.~\cite{PathakShapiro:JMAA77}. 

We then compute $T(M)$. For $t \in T(M)$ there is $u \in M^u$ such that $\sigma_t^{\psi\circ E_1}=\mathrm{Ad}u$. By (a) we have $u \in M_1^u$. Using the Connes Radon--Nikodym cocycle $[D\varphi_1:D\psi]_t$ (\cite[\S VIII.3, Corollary IX.4.22]{Takesaki:Book2}) one has $\sigma_t^\varphi = \mathrm{Ad}[D\varphi_1:D\psi]_t\circ\sigma_t^{\psi\circ E_1} = \mathrm{Ad}[D\varphi_1:D\psi]_t u$ and $[D\varphi_1:D\psi]_t u \in M_1^u$. Hence, by the argument in \cite[Corollary 8]{Ueda:MathScand01} we get $[D\varphi_1:D\psi]_t u \in \mathbb{C}1$ so that $\sigma_t^\varphi = \mathrm{Id}$. Hence $T(M) = \{t \in \mathbb{R}\,|\,\sigma_t^\varphi = \mathrm{Id}\} = \{ t \in \mathbb{R}\,|\,\sigma_t^{\varphi_1} = \mathrm{Id} = \sigma_t^{\varphi_2} \}$. 

Finally, notice that the above T-set formula shows that $M$ is semifinite if and only if $\varphi$ is tracial. Thus it is impossible that $M$ becomes of type II$_\infty$ or type I$_\infty$. It is also impossible that $M$ becomes of type I$_n$ due to its infinite dimensionality.             
\end{proof} 
 
We will then prove that the free product von Neumann algebra $M$ satisfies $M'\cap M^\omega = \mathbb{C}$ under the same hypothesis as in Theorem \ref{T-3.4}. Let $M^\omega$, $M_1^\omega$ and $M_2^\omega$ be the ultraproducts of $M$, $M_1$ and $M_2$, respectively. Thanks to Lemma \ref{L-2.1}, $M_1^\omega$ and $M_2^\omega$ can naturally be regarded as von Neumann subalgebras of $M^\omega$, and $E_i^\omega : M^\omega \rightarrow M_i^\omega$ denotes the canonical lifting of $E_i$, $i=1,2$, see \S\S2.2. 

\begin{proposition}\label{P-3.5} If there exist a faithful normal state $\psi$ on $M_1$ and $u,v \in ((M_1)_\psi)^u$ such that $\varphi_1(u^n) = \delta_{n,0}$ and $\psi(v^n) = \delta_{n,0}$ for $n \in \mathbb{Z}$, then for any $x \in \{u,v\}'\cap M^\omega$ and any $y^\circ \in M_2^\circ$ one has $\Vert y^\circ(x-E_1^\omega(x))\Vert_{(\psi\circ E_1)^\omega} \leq \Vert [x,y^\circ] \Vert_{(\psi\circ E_1)^\omega}$. 
\end{proposition}  
\begin{proof} 
As in \cite[Proposition 5]{Ueda:TAMS03} an estimate technique used below is essentially borrowed from \cite[Lemma 2.1]{Popa:AdvMath83}, but several additional, technical difficulties occur because we use $\psi\circ E_1$ instead of the free product state $\varphi$. 

In what follows, write $M_1^\triangledown := \mathrm{Ker}(\psi)$. It is not hard to see that $\mathrm{span}(\Lambda^\circ(M_1^\circ,M_2^\circ)\setminus M_1^\circ)$ coincides with the linear span of the following sets of words: 
\begin{equation}\label{Eq-3.1} 
M_1^\circ M_2^\circ \cdots M_1^\triangledown, \quad 
M_1^\circ \cdots M_2^\circ, \quad 
M_2^\circ \cdots M_1^\triangledown, \quad 
M_2^\circ \cdots M_2^\circ 
\end{equation} 
by using the decompositions $x \in M_i \mapsto \varphi_i(x)1+(x-\varphi_i(x)1) \in \mathbb{C}1+M_i^\circ$ ($i=1,2$) and $x \in M_1 \mapsto \psi(x)1+(x-\psi(x)1) \in \mathbb{C}1+M_1^\triangledown$. Here and in the rest of this paper we denote, for example, by $M_1^\circ M_2^\circ \cdots M_1^\triangledown$ the set of all words $x_1^\circ x_2^\circ \cdots x_{2n}^\circ y^\triangledown$ with the following properties: $n$ {\it can be} an arbitrary natural number, both $x_{2\ell-1}^\circ \in M_1^\circ$ and $x_{2\ell}^\circ \in M_2^\circ$ hold for every $\ell = 1,\dots,n$ and moreover $y^\triangledown \in M_1^\triangledown$.  
We easily have  
\begin{align*} 
(M_1^\circ \cdots M_2^\circ)^*(M_1^\circ M_2^\circ \cdots M_1^\triangledown) 
&\subset M_1^\triangledown + \mathrm{Ker}(E_1), \\
(M_2^\circ \cdots M_1^\triangledown)^*(M_1^\circ M_2^\circ \cdots M_1^\triangledown) &\subset \mathrm{Ker}(E_1), \\
(M_2^\circ \cdots M_2^\circ)^*(M_1^\circ M_2^\circ \cdots M_1^\triangledown) 
&\subset \mathrm{Ker}(E_1),\\
(M_2^\circ \cdots M_1^\triangledown)^*(M_1^\circ \cdots M_2^\circ) 
&\subset \mathrm{Ker}(E_1),\\
(M_2^\circ \cdots M_2^\circ)^*(M_1^\circ \cdots M_2^\circ) 
&\subset \mathrm{Ker}(E_1),\\
(M_2^\circ \cdots M_2^\circ)^*(M_2^\circ \cdots M_1^\triangledown) 
&\subset M_1^\triangledown + \mathrm{Ker}(E_1)
\end{align*}
by using the decompositions $x \in M_i \mapsto \varphi_i(x)1+(x-\varphi_i(x)1) \in \mathbb{C}1+M_i^\circ$ ($i=1,2$) again and again. Hence the four sets of words in \eqref{Eq-3.1} are mutually orthogonal in $L^2(M,\psi\circ E_1)$. Write $\mathcal{H} := L^2(M,\psi\circ E_1)$ for simplicity and denote by $\mathcal{X}_1,\mathcal{X}_2,\mathcal{X}_3,\mathcal{X}_4$ the closed subspaces of $\mathcal{H}$ generated by the sets in \eqref{Eq-3.1}, respectively, inside $\mathcal{H}$ via $\Lambda_{\psi\circ E_1}$. Then we see $\mathcal{H} = \overline{\Lambda_{\psi\circ E_1}(M_1)}\oplus \mathcal{X}_1\oplus\mathcal{X}_2\oplus\mathcal{X}_3\oplus\mathcal{X}_4$ as in the proof of Proposition \ref{P-3.1}. Denote by $P_k$, $k=1,2,3,4$, the projection from $\mathcal{H}$ onto $\mathcal{X}_k$, and clearly 
\begin{equation}\label{Eq-3.2} 
\Big(I_{\mathcal{H}}-\sum_{k=1}^4 P_k\Big)\Lambda_{\psi\circ E_1}(x) = \Lambda_{\psi\circ E_1}(E_1(x)), \quad x \in M. 
\end{equation}   

Define two unitary operators $S,T$ on $\mathcal{H}$ by 
\begin{equation*} 
S\Lambda_{\psi\circ E_1}(x) := \Lambda_{\psi\circ E_1}(uxu^*), \quad T\Lambda_{\psi\circ E_1}(x) := \Lambda_{\psi\circ E_1}(vxv^*)
\end{equation*} 
for $x \in M$. Here is a simple claim. 

\begin{claim}\label{C-3.6} We have{\rm:} \\
{\rm (3.6.1)} $\{S^n\mathcal{X}_k\}_{n\in\mathbb{Z}}$ is a family of mutually orthogonal subspaces for $k=3,4$. \\
{\rm (3.6.2)} $\{T^n\mathcal{X}_2\}_{n\in\mathbb{Z}}$ is family of mutually orthogonal subspaces. 
\end{claim}
\begin{proof} (3.6.1) is shown easily, so left to the reader. (3.6.2) follows from 
\begin{align*} 
&(v^n(M_1^\circ\cdots M_2^\circ)v^{-n})^*(v^m(M_1^\circ\cdots M_2^\circ)v^{-m}) \\
&\quad\quad = 
v^n M_2^\circ \cdots (M_1^\circ v^{m-n} M_1^\circ) \cdots M_2^\circ v^{-m} 
\subset \mathbb{C}v^{n-m} + v^n\mathrm{Ker}(E_1)v^{-m}. 
\end{align*} 
In fact, one easily observes $M_2^\circ \cdots (M_1^\circ v^{m-n} M_1^\circ) \cdots M_2^\circ \subset \mathbb{C}1+\mathrm{span}(\Lambda^\circ(M_1^\circ,M_2^\circ)\setminus M_1^\circ)$ by using the decompositions $x \in M_i \mapsto \varphi_i(x)1+(x-\varphi_i(x)1) \in \mathbb{C}1+M_i^\circ$ ($i=1,2$) again and again.  
\end{proof} 

Choose and fix arbitrary $x \in \{u,v\}'\cap M^\omega$, and let $(x(m))_m$ be its representative. For each $n \in \mathbb{Z}$ we have 
\begin{align*}
\lim_{m\rightarrow\omega}\big\Vert\Lambda_{\psi\circ E_1}(x(m)-u^n x(m)u^{-n})\big\Vert_{\psi\circ E_1} 
&= 
\big\Vert\Lambda_{(\psi\circ E_1)^\omega}(x-u^n xu^{-n})\big\Vert_{(\psi\circ E_1)^\omega} = 0, \\
\lim_{m\rightarrow\omega}\big\Vert\Lambda_{\psi\circ E_1}(x(m)-v^n x(m)v^{-n})\big\Vert_{\psi\circ E_1} 
&= 
\big\Vert\Lambda_{(\psi\circ E_1)^\omega}(x-v^n xv^{-n})\big\Vert_{(\psi\circ E_1)^\omega} = 0. 
\end{align*} 
Thus, for each $\varepsilon>0$ and each $n_0 \in \mathbb{N}$ there is a neighborhood $W$ in $\beta(\mathbb{N})$ at $\omega$ so that 
\begin{equation}\label{Eq-3.3}
\Vert \Lambda_{\psi\circ E_1}(x(m)-u^n x(m)u^n)\Vert_{\psi\circ E_1} < \varepsilon, \quad 
\Vert \Lambda_{\psi\circ E_1}(x(m)-v^n x(m)v^n)\Vert_{\psi\circ E_1} < \varepsilon   
\end{equation} 
for every $n \in \mathbb{N}$ with $|n|\leq n_0$ and for every $m \in W\cap\mathbb{N}$. For $k=3,4$ and for every $m \in W\cap\mathbb{N}$ we have 
\begin{equation*}
\begin{aligned} 
&\big\Vert P_k\Lambda_{\psi\circ E_1}(x(m))\big\Vert_{\psi\circ E_1}^2 \\
&= 
\frac{1}{2n_0+1} \sum_{n=-n_0}^{n_0} \big\Vert S^n P_k\Lambda_{\psi\circ E_1}(x(m))\big\Vert_{\psi\circ E_1}^2 \\
&\leq 
\frac{1}{2n_0+1}\sum_{n=-n_0}^{n_0}2\Big( 
\big\Vert S^n P_k\Lambda_{\psi\circ E_1}(x(m)) - 
S^n P_k S^{-n}\Lambda_{\psi\circ E_1}(x(m))\big\Vert_{\psi\circ E_1}^2 \\ 
&\phantom{aaaaaaaaaaaaaaaaaaaaaaaaaaaaaa}+ \big\Vert  
S^n P_k S^{-n}\Lambda_{\psi\circ E_1}(x(m))\big\Vert_{\psi\circ E_1}^2\Big) \\
&=\frac{2}{2n_0+1}\sum_{n=-n_0}^{n_0}\Big( 
\big\Vert S^n P_k S^{-n}\Lambda_{\psi\circ E_1}(u^n x(m) u^{-n} - x(m))\big\Vert_{\psi\circ E_1}^2 \\ 
&\phantom{aaaaaaaaaaaaaaaaaaaaaaaaaaaaaa}+ \big\Vert  
S^n P_k S^{-n}\Lambda_{\psi\circ E_1}(x(m))\big\Vert_{\psi\circ E_1}^2\Big) \\
&\leq 
\frac{2}{2n_0+1}\sum_{n=-n_0}^{n_0}\Big(\varepsilon^2 + \big\Vert  
S^n P_k S^{-n}\Lambda_{\psi\circ E_1}(x(m))\big\Vert_{\psi\circ E_1}^2\Big) \\
&\leq 2\varepsilon^2 + \frac{2}{2n_0+1}\big\Vert\Lambda_{\psi\circ E_1}(x(m))\big\Vert_{\psi\circ E_1}^2 \\
&\leq 2\varepsilon^2 + \frac{2}{2n_0+1}\Vert(x(m))_m\Vert_\infty^2, 
\end{aligned}
\end{equation*}  
where the first equality is due to the unitarity of $S$, the second inequality is obtained by the parallelogram identity, the fourth is due to \eqref{Eq-3.3}, and the fifth comes from (3.6.1) together with the fact that $S^n P_k S^{-n}$ is the projection onto $S^n \mathcal{X}_k$. The exactly same argument with replacing $S$ and (3.6.1) by $T$ and (3.6.2) shows 
\begin{equation*}
\Vert P_2\Lambda_{\psi\circ E_1}(x(m))\Vert_{\psi\circ E_1}^2 < 
2\varepsilon^2 + \frac{2}{2n_0+1}\Vert(x(m))_m\Vert_\infty^2
\end{equation*} 
for every $m \in W\cap\mathbb{N}$. Consequently, for each $\delta>0$ there is a neighborhood $W_\delta$ in $\beta(\mathbb{N})$ at $\omega$ such that 
\begin{equation}\label{Eq-3.4}
\Big\Vert(P_2+P_3+P_4)\Lambda_{\psi\circ E_1}(x(m))\Big\Vert_{\psi\circ E_1} < \delta 
\end{equation} 
as long as $m \in W_\delta\cap\mathbb{N}$. 

We then regard $L^2(M^\omega,(\psi\circ E_1)^\omega)$ as a closed subspace of the ultraproduct $\mathcal{H}^\omega$ as explained in \S\S2.2. Choose and fix arbitrary $y^\circ \in M_2^\circ$. We have 
\begin{align*} 
&\Big\Vert\Lambda_{(\psi\circ E_1)^\omega}(y^\circ(x-E_1^\omega(x)))-\Big[\big(y^\circ P_1\Lambda_{\psi\circ E_1}(x(m))\big)_m\Big]\Big\Vert_{\mathcal{H}^\omega} \\
&= 
\lim_{m\rightarrow\omega} 
\big\Vert \Lambda_{\psi\circ E_1}(y^\circ(x(m)-E_1(x(m)))) - y^\circ P_1\Lambda_{\psi\circ E_1}(x(m))\big\Vert_{\psi\circ E_1} \\
&\leq 
\sup_{m\in W_\delta\cap\mathbb{N}} \big\Vert \Lambda_{\psi\circ E_1}(y^\circ(x(m)-E_1(x(m)))) - y^\circ P_1\Lambda_{\psi\circ E_1}(x(m))\big\Vert_{\psi\circ E_1} \\
&\leq \Vert y^\circ\Vert_\infty \sup_{m\in W_\delta\cap\mathbb{N}} \big\Vert (P_2+P_3+P_4)\Lambda_{\psi\circ E_1}(x(m))\big\Vert_{\psi\circ E_1} \quad \text{(use \eqref{Eq-3.2})} \\
&< \Vert y^\circ\Vert_\infty \delta 
\end{align*}
by \eqref{Eq-3.4}. Since $\delta>0$ is arbitrary, one has, in $\mathcal{H}^\omega$, 
\begin{equation*} 
\Lambda_{(\psi\circ E_1)^\omega}(y^\circ(x-E_1^\omega(x))) = \Big[\big(y^\circ P_1\Lambda_{\psi\circ E_1}(x(m))\big)_m\Big].
\end{equation*} 
Also it is trivial that 
\begin{equation*}
\Lambda_{(\psi\circ E_1)^\omega}(y^\circ E_1^\omega(x) - E_1^\omega(x)y^\circ) = 
\Big[\big(\Lambda_{\psi\circ E_1}(y^\circ E_1(x(m)) - E_1(x(m))y^\circ\big)_m\Big]. 
\end{equation*}
Set 
\begin{align*} 
y_n^\circ :=&\, \frac{1}{\sqrt{n\pi}}\int_{-\infty}^{+\infty} e^{-t^2/n}\sigma_t^{\psi\circ E_1}(y^\circ)\,dt \\
=&\, 
\frac{1}{\sqrt{n\pi}}\int_{-\infty}^{+\infty} e^{-t^2/n} [D\psi:D\varphi_1]_t\,\sigma_t^{\varphi_2}(y^\circ)\,[D\psi:D\varphi_1]_t^*\,dt, 
\end{align*} 
which falls into the $\sigma$-weak or equivalently $\sigma$-strong closure of the linear span of $M_1 M_2^\circ M_1$. Remark (see \cite[Lemma VIII.2.3]{Takesaki:Book2}) that $t \mapsto \sigma_t^{\psi\circ E_1}(y_n^\circ)$ is extended to an entire function, still denoted by $\sigma_z^{\psi\circ E_1}(y_n^\circ)$, $z \in \mathbb{C}$, for every $n\in\mathbb{N}$ and $y_n^\circ \longrightarrow y^\circ$ $\sigma$-weakly as $n\rightarrow\infty$. For each fixed $n$ we have  
\begin{align*} 
&\Big\Vert\Lambda_{(\psi\circ E_1)^\omega}((x-E_1^\omega(x))y^\circ_n)-\Big[\big(J\sigma_{-i/2}^{\psi\circ E_1}(y_n^\circ)^* J P_1\Lambda_{\psi\circ E_1}(x(m))\big)_m\Big]\Big\Vert_{\mathcal{H}^\omega} \\
&= 
\lim_{m\rightarrow\omega} 
\big\Vert \Lambda_{\psi\circ E_1}((x(m)-E_1(x(m)))y^\circ_n) -  J\sigma_{-i/2}^{\psi\circ E_1}(y_n^\circ)^* J P_1\Lambda_{\psi\circ E_1}(x(m))\big\Vert_{\psi\circ E_1} \\
&= 
\lim_{m\rightarrow\omega} 
\big\Vert J\sigma_{-i/2}^{\psi\circ E_1}(y_n^\circ)^* J\big(\Lambda_{\psi\circ E_1}(x(m)-E_1(x(m))) -  P_1\Lambda_{\psi\circ E_1}(x(m))\big)\big\Vert_{\psi\circ E_1} \\
&\leq 
\sup_{m\in W_\delta\cap\mathbb{N}} \big\Vert J\sigma_{-i/2}^{\psi\circ E_1}(y_n^\circ)^* J\big(\Lambda_{\psi\circ E_1}(x(m)-E_1(x(m))) -  P_1\Lambda_{\psi\circ E_1}(x(m))\big)\big\Vert_{\psi\circ E_1} \\
&< \Vert J\sigma_{-i/2}^{\psi\circ E_1}(y_n^\circ)^* J\Vert_\infty \delta 
\end{align*}
as before by \eqref{Eq-3.2},\eqref{Eq-3.4}, where $J$ is the modular conjugation of $M \curvearrowright \mathcal{H} = L^2(M,\psi\circ E_1)$ and we used \cite[Lemma VIII.3.10]{Takesaki:Book2}. Hence we get 
\begin{equation*}
\Lambda_{(\psi\circ E_1)^\omega}((x-E_1^\omega(x))y^\circ_n)=\Big[\big(J\sigma_{-i/2}^{\psi\circ E_1}(y_n^\circ)^* J P_1\Lambda_{\psi\circ E_1}(x(m))\big)_m\Big]
\end{equation*}
in $\mathcal{H}^\omega$. 
Note that $M_1^\triangledown y_n^\circ$ sits in the $\sigma$-strong closure of the linear span of $M_1^\triangledown M_1 M_2^\circ M_1\,(\subset M_1 M_2^\circ M_1)$. Then we observe that 
\begin{align*} 
y^\circ P_1\Lambda_{\psi\circ E_1}(x(m)) &\in 
\overline{\mathrm{span}\,\Lambda_{\psi\circ E_1}(\underbrace{M_2^\circ M_1^\circ M_2^\circ \cdots M_1^\triangledown}_{\text{length}\, \geq 4})}
\end{align*} 
is orthogonal to 
\begin{align*}
\Lambda_{\psi\circ E_1}(y^\circ E_1(x(m)) - E_1(x(m))y^\circ) 
&\in 
\Lambda_{\psi\circ E_1}(M_2^\circ M_1 - M_1 M_2^\circ) \\
&\subset \overline{\Lambda_{\psi\circ E_1}(M_2^\circ)}\oplus\overline{\mathrm{span}\,\Lambda_{\psi\circ E_1}(M_2^\circ M_1^\triangledown)} \\ 
&\phantom{aaaaaaaaaaaaaaaaa}\oplus\overline{\mathrm{span}\,\Lambda_{\psi\circ E_1}(M_1^\circ M_2^\circ)}, \\
J\sigma_{-i/2}^{\psi\circ E_1}(y_n^\circ)^* J P_1\Lambda_{\psi\circ E_1}(x(m)) 
&\in 
J\sigma_{-i/2}^{\psi\circ E_1}(y_n^\circ)^* J\cdot\overline{\mathrm{span}\,\Lambda_{\psi\circ E_1}(M_1^\circ M_2^\circ \cdots M_1^\triangledown)} \\
&\subset 
\overline{\mathrm{span}\,\Lambda_{\psi\circ E_1}(M_1^\circ M_2^\circ \cdots M_1^\triangledown y_n^\circ)} \\
&\subset 
\overline{\Lambda_{\psi\circ E_1}(M_1)} \oplus \overline{\mathrm{span}\,\Lambda_{\psi\circ E_1}(\underbrace{M_1^\circ M_2^\circ \cdots}_{\text{length}\,\geq 2})},   
\end{align*}
which can be checked by using the decompositions $x \in M_i \mapsto \varphi_i(x)1+(x-\varphi_i(x)1) \in \mathbb{C}1+M_i^\circ$ ($i=1,2$) and $x \in M_1 \mapsto \psi(x)1+(x-\psi(x)1) \in \mathbb{C}1+M_1^\triangledown$.          
Hence we conclude that $\Lambda_{(\psi\circ E_1)^\omega}(y^\circ(x-E_1^\omega(x)))$ is orthogonal to $\Lambda_{(\psi\circ E_1)^\omega}(y^\circ E_1^\omega(x) - E_1^\omega(x)y^\circ)$ and $\Lambda_{(\psi\circ E_1)^\omega}((x-E_1^\omega(x))y^\circ_n)$. Moreover, 
\begin{align*} 
&\big(\Lambda_{(\psi\circ E_1)^\omega}((x-E_1^\omega(x))y^\circ)\big|\Lambda_{(\psi\circ E_1)^\omega}(y^\circ(x-E_1^\omega(x)))\big)_{(\psi\circ E_1)^\omega} \\
&\quad\quad= \lim_{n\rightarrow\infty}\big(\Lambda_{(\psi\circ E_1)^\omega}((x-E_1^\omega(x))y^\circ_n)\big|\Lambda_{(\psi\circ E_1)^\omega}(y^\circ(x-E_1^\omega(x)))\big)_{(\psi\circ E_1)^\omega} = 0
\end{align*} 
since $y_n^\circ \longrightarrow y^\circ$ $\sigma$-weakly, as $n \rightarrow \infty$. Therefore, 
\begin{align*} 
\big\Vert y^\circ(x-E_1^\omega(x)) \big\Vert_{(\psi\circ E_1)^\omega} 
\leq 
\big\Vert y^\circ x - x y^\circ\big\Vert_{(\psi\circ E_1)^\omega}.
\end{align*} 
Hence we are done.   
\end{proof} 

As in Theorem \ref{T-3.4} the previous proposition implies the next theorem. 

\begin{theorem}\label{T-3.7} If $M_1\neq\mathbb{C}\neq M_2$ and if either $M_1$ or $M_2$ is diffuse, then $M'\cap M^\omega = \mathbb{C}$. Also, if either $(M_1)_{\varphi_1}$ or $(M_2)_{\varphi_2}$ is diffuse and the other $(M_i)_{\varphi_i} \neq \mathbb{C}$, then $(M_\varphi)'\cap M^\omega = \mathbb{C}$.  
\end{theorem}  
\begin{proof} We may and do assume that $M_1$ is diffuse. As we saw in Theorem \ref{T-3.4}, there is a faithful normal state $\psi$ on $M_1$ whose centralizer contains a diffuse von Neumann subalgebra $A$. Since $A$ is diffuse, one can construct two von Neumann subalgebras $B_1, B_2$ of $A$ in such a way that $(B_1,\varphi_1|_{B_1})$ and $(B_2,\psi|_{B_2})$ are copies of  the infinite tensor product of $\mathbb{C}^2$ with the equal weight state $(1/2,1/2)$, which is naturally isomorphic to $(L^\infty(\mathbb{T}),\int_\mathbb{T}\,\cdot\,\mu_\mathbb{T}(d\zeta))$ with the Haar probability measure $\mu_\mathbb{T}$ on the $1$-dimensional torus $\mathbb{T}$ (see e.g.~\cite[Theorem III.1.22]{Takesaki:Book1}). Hence one can find $u,v \in A^u$ in such a way that $\varphi_1(u^n) = \delta_{n,0}$ and $\psi(v^n) = \delta_{n,0}$ for $n \in \mathbb{Z}$. Since $M_2 \neq \mathbb{C}$, there are two orthogonal non-zero $e_1,e_2 \in M_2^p$ with $e_1+e_2=1$. Set $y^\circ := \varphi_2(e_2)e_1 - \varphi_2(e_1)e_2 \in M_2^\circ$, and by Proposition \ref{P-3.5} we get $\{u,v,y^\circ\}' \cap M^\omega \subseteq M_1^\omega$, since $y^\circ$ is invertible. It is easy to see that $M_1^\omega$ and $M_2^\omega$ are free with respect to $\varphi^\omega$ (see e.g.~\cite[Proposition 4]{Ueda:TAMS03}). Hence we conclude $\{u,v,y^\circ\}' \cap M^\omega=\mathbb{C}1$ because $(M_1^\omega)^\circ y^\circ, y^\circ(M_1^\omega)^\circ$ are orthogonal in $L^2(M^\omega,\varphi^\omega)$. 

The last assertion follows from the same argument as above. In fact, we may and do assume that $(M_1)_{\varphi_1}$ is diffuse and $(M_2)_{\varphi_2} \neq \mathbb{C}$. Then the above $\psi$ is chosen as $\varphi_1$ itself and consequently $u=v \in (M_1)_{\varphi_1}$. Moreover one can choose $y^\circ$ from $(M_2)_{\varphi_2}$. Thus the above argument shows that $(M_\varphi)'\cap M^\omega \subseteq \{u=v, y^\circ\}'\cap M^\omega = \mathbb{C}$.          
\end{proof} 

\section{Main Theorem} 

\subsection{Statements} Let $M_1$ and $M_2$ be arbitrary $\sigma$-finite von Neumann algebras and $\varphi_1$ and $\varphi_2$ be arbitrary faithful normal states on them, respectively. For the questions mentioned in \S1 we may and do assume that $M_1\neq\mathbb{C}\neq M_2$ and $\mathrm{dim}(M_1)+\mathrm{dim}(M_2) \geq 5$. Otherwise, the resulting free product von Neumann algebra $M$ is either $M_1$ (when $M_2=\mathbb{C}$), $M_2$ (when $M_1=\mathbb{C}$) or $(\mathbb{C}\oplus\mathbb{C})\star (\mathbb{C}\oplus\mathbb{C})$ whose structure is explicitly determined by the structure theorem on two freely independent projections, see \cite[Example 3.6.7]{VDN} and \cite[Theorem 1.1]{Dykema:DukeMathJ93}. We can write $M_i = M_{i,d}\oplus M_{i,c}$ ($i=1,2$), where 
$$
M_{i,d} = \sideset{}{^\oplus}\sum_{j\in J_i} B(\mathcal{H}_{ij}) \quad \text{or} \quad  = 0 \quad \text{possibly with $d_{ij} := \mathrm{dim}(\mathcal{H}_{ij}) = \infty$}
$$
and $M_{i,c}$ is diffuse or $M_{i,c}=0$. We can then choose a matrix unit system $\{e^{(ij)}_{st}\}_{s,t}$ of $B(\mathcal{H}_{ij})$ that diagonalizes the density operator of $\varphi_i|_{B(\mathcal{H}_{ij})}$, that is, 
$$
\varphi_i|_{B(\mathcal{H}_{ij})} = \mathrm{Tr}\big(\big(\sum_{s=1}^{d_{ij}} \lambda_s^{(ij)} e_{ss}^{(ij)}\big)\,\cdot\,\big)
$$
with $\lambda_1^{(ij)} \geq \lambda_2^{(ij)} \geq \cdots$. 

With these notations the main theorem of this section is stated as follows. 

\begin{theorem}\label{T-4.1} Under the above assumption the resulting free product von Neumann algebra $M$ of $(M_1,\varphi_1)$ and $(M_2,\varphi_2)$ is of the form $M_d\oplus M_c$ possibly with $M_d = 0$, where $M_d$ is a multi-matrix algebra and $M_c$ is a factor of type II$_1$ or III$_\lambda$ with $\lambda\neq0$ whose T-set $T(M_c)$ is computed as $\{t \in \mathbb{R}\,|\,\sigma_t^{\varphi_1}=\mathrm{Id}=\sigma_t^{\varphi_2}\}$. Moreover $M_c$  always satisfies $M_c'\cap M_c^\omega = \mathbb{C}$, and hence $(M_c)_\omega = \mathbb{C}$.  

The explicit description of the multi-matrix part $M_d$ is as follows. If the supremum of all  $\varphi_i(e)$ with minimal $e \in \mathcal{Z}(M_{i,d})^p$, $i = 1,2$, is attained by only one minimal $p \in \mathcal{Z}(M_{i_0,d})^p$ with $i_0 \in \{1,2\}$ such that $M_{i_0,d}\,p = \mathbb{C}p$, and moreover if 
\begin{equation}\label{Eq-4.1}
J_{i'_0}^\circ := \Big\{ j \in J_{i'_0}\,\Big|\,\sum_{s=1}^{d_{i'_0 j}}\frac{1}{\lambda_s^{(i'_0 j)}} \lneqq \frac{1}{1-\varphi_{i_0}(p)}\Big\} \neq \emptyset
\end{equation}
with $\{i'_0\} = \{1,2\}\setminus\{i_0\}$, then the multi-matrix part $M_d$ is isomorphic to 
\begin{equation}\label{Eq-4.2}
\sideset{}{^\oplus}\sum_{j\in J_{i'_0}^\circ} B(\mathcal{H}_{i'_0 j})
\end{equation}
and the copy of $B(\mathcal{H}_{i'_0 j})$ inside $M_d$ is given by the matrix unit system
\begin{equation}\label{Eq-4.3}
f_{st}^{(i'_0 j)} := e_{sd}^{(i'_0 j)}\big(\bigwedge_{t'=1}^{d} e_{dt'}^{(i'_0 j)}(p\wedge e_{t' t'}^{(i'_0 j)})e_{t'd}^{(i'_0 j)}\big)e_{dt}^{(i'_0 j)}, \quad 1 \leq s,t \leq d := d_{i'_0 j} 
\end{equation}
{\rm(}n.b.~$d = d_{i'_0 j}$ must be finite due to \eqref{Eq-4.1}{\rm)}. Moreover the free product state $\varphi$ satisfies 
\begin{equation}\label{Eq-4.4} 
\varphi(f_{st}^{(i'_0 j)}) = \delta_{st}\,\lambda_s^{(i'_0 j)}\Big(1-(1-\varphi_{i_0}(p))\sum_{r=1}^{d_{i'_0 j}}\frac{1}{\lambda_r^{(i'_0 j)}}\Big). 
\end{equation} 
Otherwise, $M_d = 0$. 
\end{theorem}

\begin{remarks}\label{R-4.2}{\rm 
(1) The only one `largest' central minimal projection $p$ in the explicit description of $M_d$ must satisfy $\varphi_{i_0}(p) \gneqq 1/2$. (2) The condition \eqref{Eq-4.1} says that only $B(\mathcal{H}_{i'_0 j})$ with $d_{i'_0 j} \leqq 1/(1-\varphi_{i_0}(p))$ may appear in the multi-matrix part $M_d$. (3) The matrix units $f_{st}^{(i'_0 j)}$'s in \eqref{Eq-4.3} is nothing less than the `meet' of $p$ and the $e_{st}^{(i'_0 j)}$'s in the sense of Dykema \cite[\S1]{Dykema:AmerJMath95}. Note that \eqref{Eq-4.3} shows that the minimal $p \in \mathcal{Z}(M_{i_0,d})^p$ in the description of $M_d$ dominates $1_{M_d}$. (4) The proof below (see \S\S4.2.3) shows a very strong `ergodicity' of $\varphi$, that is, $((M_c)_{\varphi|_{M_c}})'\cap M_c^\omega = \mathbb{C}$, at least when both $M_1=M_{1,d}$ and $M_2=M_{2,d}$ hold. (5) Theorem \ref{T-4.1} completes to show the following expected fact: the given `non-trivial' free product von Neumann algebra $M$ is `amenable' if and only if both $M_1$ and $M_2$ are $2$ dimensional, which is analogous to the well-known fact that only $\mathbb{Z}_2\star\mathbb{Z}_2$ becomes amenable among the free product groups.}     
\end{remarks}

As mentioned in \S1 some recent results in Popa's deformation/rigidity theory due to Chifan--Houdayer and Houdayer--Ricard enable us to give more facts on the diffuse factor part $M_c$. 

\begin{corollary}\label{C-4.3} If $M_1$ and $M_2$ have separable preduals, then the diffuse factor part $M_c$ always has the following properties{\rm:} {\rm(1)} $M_c$ is prime, that is, there is no pair $P_1, P_2$ of diffuse factors such that $M_c = P_1\bar{\otimes}P_2$. {\rm(2)} If the given $M_1$ and $M_2$ are hyperfinite {\rm(}or amenable{\rm)}, then any non-hyperfinite von Neumann subalgebra of $M_c$ that is the range of a faithful normal conditional expectation from $M_c$ has no Cartan subalgebra.   
\end{corollary}
\begin{proof} When $M=M_c$, \cite[Theorem 5.2]{ChifanHoudayer:DukeMathJ10} directly implies the assertion (1) since we have known that $M$ is a full factor (and hence not hyperfinite). When $M\neq M_c$, the same proof of \cite[Theorem 5.2]{ChifanHoudayer:DukeMathJ10} with replacing the projection $p \in L(\mathbb{R})$ there by $1_{M_c}\otimes p \in \mathcal{Z}(M)\bar{\otimes}L(\mathbb{R})$ (which sits inside the continuous core of $M$) works well for showing that $M_c$ is prime, since $M_c$ is a full factor. The assertion (2) on $M_c$ is easily derived from \cite[Theorem 5.4 (2)]{HoudayerRicard:Preprint10} as follows. Suppose that a given non-hyperfinite von Neumann subalgebra $N$ of $M_c$ which is the range of a faithful normal conditional expectation $E_N$ from $M_c$ has a Cartan subalgebra $A$. We can choose a Cartan subalgebra $B$ of $M_d$ since $M_d$ is a multi-matrix algebra. It is clear that $B\oplus A$ becomes a Cartan subalgebra of $M_d\oplus N$, a von Neumann subalgebra of $M$, which is the range of $\mathrm{Id}_{M_d}\oplus E_N$. Hence, applying \cite[Theorem 5.4]{HoudayerRicard:Preprint10} to $B\oplus A \subset M_d\oplus N$ we conclude that $M_d\oplus N$ must be hyperfinite under the assumption of the assertion (2). But this contradicts the non-hyperfiniteness of $N$. \end{proof}

Remark that the statement of the above (2) is almost valid true even when $M$ {\rm(}or $M_c${\rm)} has the weak$^*$ complete metric approximation property {\rm(}see e.g.~\cite[Definition 2.9]{OzawaPopa:AnnMath10}{\rm)}; under that assumption the diffuse factor part $M_c$ has no Cartan subalgebra due to \cite[Theorem 5.4 (1)]{HoudayerRicard:Preprint10}.   A very recent work due to Ozawa \cite{Ozawa:Preprint10} makes it hold under the weaker assumption that $M$ {\rm(}or $M_c${\rm)} has the weak$^*$ completely bounded approximation property.  

\subsection{Proof of Theorem \ref{T-4.1}}  
This subsection is entirely devoted to the proof of Theorem \ref{T-4.1}. We start with one simple fact which will repeatedly be used in the proof without any claim. If a given (unital) inclusion of von Neumann algebras, say $N_1 \subseteq N_2$, satisfies that $(pN_1 p)'\cap(pN_2 p) = N_1' p \cap (pN_2 p) = \mathbb{C}p$ and $c_{N_1}(p) = 1$ for some non-zero $p \in N_1^p$, then the original inclusion is stably isomorphic to $pN_1 p \subseteq pN_2 p$ and hence $N_1'\cap N_2 = \mathbb{C}$. In fact, choose arbitrary $f \in \mathcal{Z}(N_1)^p$. Then $fp = pfp$ falls into $\mathcal{Z}(pN_1 p)^p$ so that $fp$ must be $p$ or $0$. If $fp=p$, then the definition of $c_{N_1}(p)$ implies $1 = c_{N_1}(p) \leq f$, i.e., $f = 1$. Also, if $fp = 0$, then $f(xp\xi) = 0$ for all $x \in N_1$ and $\xi \in \mathcal{H}$, a Hilbert space on which $N_1$ acts. This shows $f = 0$, since $N_1 p\mathcal{H}$ is a dense subspace of $c_{N_1}(p)\mathcal{H}$ and $c_{N_1}(p) = 1$. Hence $N_1$ is a factor. Therefore we conclude that $pN_1 p \subseteq pN_2 p$ is stably isomorphic to the original $N_1 \subseteq N_2$ in the same way as in the proof of the fact that the property of `$N'\cap N^\omega = \mathbb{C}$' is a stably isomorphic one (see \S\S2.2).  

\medskip
If $M_{1,d}=0$ or $M_{2,d}=0$, then the desired assertions immediately follow from Theorem \ref{T-3.4}, Theorem \ref{T-3.7} in \S3. Hence we need to deal with only the following cases: 
\begin{itemize}
\item[(a)] all $M_{1,d}, M_{1,c}$, $M_{2,d}$ and $M_{2,c}$ are not $0$.  
\item[(b)] $M_{1,d}$, $M_{1,c}$ and $M_{2,d}$ are not $0$, but $M_{2,c} = 0$. 
\item[(c)] $M_{1,d}$, $M_{2,d}$ and $M_{2,c}$ are not $0$, but $M_{1,c} = 0$. 
\item[(d)] both $M_{1,c}$ and $M_{2,c}$ are $0$. 
\end{itemize}   
Lemma \ref{L-2.2} enables us to reduce the cases (a),(b),(c) to the case of $M_{1,d}=0$ or $M_{2,d}=0$ and the case (d). Thus, if the case (d) is assumed to be confirmed already, then one can easily treat the other cases as follows. 

\subsubsection{The proof of the cases {\rm(b),(c)}.} By switching $M_1$ and $M_2$ if necessary it suffices to deal with only (b). Consider the inclusion $M \supset N:=(M_{1,d}\oplus\mathbb{C}1_{M_{1,c}})\vee M_2$. Clearly $(N,\varphi|_N)$ is the free product 
\begin{equation}\label{Eq-4.5}
(M_{1,d}\oplus\mathbb{C}1_{M_{1,c}},\varphi_1|_{M_{1,d}\oplus\mathbb{C}1_{M_{1,c}}})\star(M_2,\varphi_2).
\end{equation} 
Moreover, by Lemma \ref{L-2.2} the pair $(1_{M_{1,c}}M1_{M_{1,c}},(1/\varphi(1_{M_{1,c}}))\varphi|_{1_{M_{1,c}}M1_{M_{1,c}}})$ is also the free product 
\begin{equation}\label{Eq-4.6}
(M_{1,c},(1/\varphi_1(1_{M_{1,c}}))\varphi_1|_{M_{1,c}})\star(1_{M_{1,c}}N1_{M_{1,c}},(1/\varphi(1_{M_{1,c}}))\varphi|_{1_{M_{1,c}}N1_{M_{1,c}}})
\end{equation} 
and $c_N(1_{M_{1,c}}) = c_M(1_{M_{1,c}})$. Note here that the free product \eqref{Eq-4.5} is in the case (d) since $M_2 = M_{2,d}$. If $\varphi_1(1_{M_{1,c}})$ is strictly greater than the supremum of all $\varphi_i(e)$ with minimal $e \in \mathcal{Z}(M_{i,d})^p$, $i=1,2$, then the central support projection $c_N(1_{M_{1,c}})$ must be $1$ since all the assertions are assumed to hold in the case (d). Hence $M$ is stably isomorphic to $1_{M_{1,c}}M1_{M_{1,c}}$ that satisfies the desired assertions except the T-set formula by Theorem \ref{T-3.4} and Theorem \ref{T-3.7} due to \eqref{Eq-4.6}. We have known that $T(M) = T(1_{M_{1,c}}M1_{M_{1,c}}) = \{ t \in \mathbb{R}\,|\,\sigma_t^\varphi|_{1_{M_{1,c}}N1_{M_{1,c}}} = \mathrm{Id} = \sigma_t^{\varphi_2} \}$, and then by assumption $\sigma_t^\varphi|_{1_{M_{1,c}}N 1_{M_{1,c}}} = \mathrm{Id}$ if and only if $\sigma_t^{\varphi_1|_{M_{1,d}}} = \mathrm{Id} = \sigma_t^{\varphi_2}$ since $1_{M_{1,c}}N1_{M_{1,c}}$ contains, as direct summand, a certain non-trivial compressed algebra of the diffuse factor part $N_c$ by a projection in $(N_c)_{\varphi|_{N_c}}$. Hence we have obtained the desired T-set formula. By the T-set formula, it is easy to observe that $M$ is never of type II$_\infty$ so that the assertion on the type of $M=M_c$ holds. The triviality of $M'\cap M^\omega$ with $M=M_c$ also holds since the property of `$N'\cap N^\omega = \mathbb{C}$' is a stably isomorphic one as remarked in \S\S2.2. 

Otherwise, that is, when $\varphi_1(1_{M_{1,c}})$ is less than the supremum of all $\varphi_i(e)$ with minimal $e \in \mathcal{Z}(M_{i,d})^p$, $i=1,2$, what we have shown in the case (d) says that $N$ is either 
\begin{itemize} 
\item $M_d\oplus\mathbb{C}q\oplus N_c$ with $c_N(1_{M_{1,c}}) = q + 1_{N_c}$, or 
\item $M_d\oplus N_c$ with $c_N(1_{M_{1,c}}) = 1_{N_c}$\item $N_c$, 
\end{itemize} 
where $M_d$ is \eqref{Eq-4.2} and $N_c$ is either a factor of type II$_1$ or III$_\lambda$ with $\lambda\neq0$, or $L^\infty(0,1)\bar{\otimes}M_2(\mathbb{C})$ if $M_{1d} = \mathbb{C}$ and $M_{2d} = \mathbb{C}\oplus\mathbb{C}$ (at this point we need the structure theorem on two freely independent projections), such that $T(N_c) = \{t \in \mathbb{R}\,|\,\sigma_t^{\varphi_1}|_{M_{1,d}} = \mathrm{Id} = \sigma_t^{\varphi_2}\}$. Then, by \eqref{Eq-4.6} one easily observes that $M = M_d\oplus M_c$ or $M = M_c$, and moreover $M_c$ is stably isomorphic to $1_{M_{1,c}}M1_{M_{1,c}}$. Therefore the exactly same argument as the previous one shows the desired assertions. Hence we are done in the cases (b),(c). 

\subsubsection{The proof of the case {\rm(a)}.} Consider the inclusion $M \supset N:=(M_{1,d}\oplus\mathbb{C}1_{M_{1,c}})\vee M_2$. Clearly $(N,\varphi|_N)$ is the free product 
\begin{equation}\label{Eq-4.7}
(M_{1,d}\oplus\mathbb{C}1_{M_{1,c}},\varphi_1|_{M_{1,d}\oplus\mathbb{C}1_{M_{1,c}}})\star(M_2,\varphi_2).
\end{equation} 
Moreover, by Lemma \ref{L-2.2} the pair $(1_{M_{1,c}}M1_{M_{1,c}},(1/\varphi(1_{M_{1,c}}))\varphi|_{1_{M_{1,c}}M1_{M_{1,c}}})$ is also the free product 
\begin{equation*}
(M_{1,c},(1/\varphi_1(1_{M_{1,c}}))\varphi_1|_{M_{1,c}})\star(1_{M_{1,c}}N1_{M_{1,c}},(1/\varphi(1_{M_{1,c}}))\varphi|_{1_{M_{1,c}}N1_{M_{1,c}}})
\end{equation*} 
and $c_N(1_{M_{1,c}}) = c_M(1_{M_{1,c}})$. Note that the free product \eqref{Eq-4.7} falls into the case (c), and therefore, by using what we have established in dealing with the cases (b),(c) we can conclude the desired assertions in the same way as in the cases (b),(c). Hence we are done in the case (a). 

\subsubsection{The proof of the case {\rm(d)}.} The proof below will essentially be done by induction together with several case-by-case arguments (and thus the proof is not so difficult analytically, though it looks complicated at a glance). The principal aim here is to prove $(M_{\varphi|_{M_c}})'\cap M_c^\omega = \mathbb{C}$ (since \cite{Dykema:FieldsInstituteComm97} does not discuss it in full generality). In the course of proving it, we will `re-prove' all the facts on the structure of $M_d$ presented in \cite{Dykema:FieldsInstituteComm97} such as \eqref{Eq-4.1}--\eqref{Eq-4.4}. The desired T-set formula of $M_c$ follows from $(M_{\varphi|_{M_c}})'\cap M_c^\omega = \mathbb{C}$ as well as the fact that every $x_{st}^{(i_0' j)} := e_{st}^{(i_0' j)} - f_{st}^{(i_0' j)}$ with $f_{st}^{(i_0' j)}$ in \eqref{Eq-4.3} defines a nonzero eigenvector in $M_c$ for $\sigma^\varphi$. Hence our proof below is independent of that given in \cite{Dykema:FieldsInstituteComm97}.     

\medskip
{\bf Step 1 -- abelian case.} We first consider the  following special case: 
$$
(M_1,\varphi_1) = \sideset{}{^\oplus}\sum_{k=1}^{K_1} \overset{p_k}{\underset{\alpha_k}{\mathbb{C}}}, \quad 
(M_2,\varphi_2) = \sideset{}{^\oplus}\sum_{k=1}^{K_1}  \overset{q_k}{\underset{\beta_k}{\mathbb{C}}} \quad \text{(possibly with $K_1,K_2 = \infty$)}
$$
(we use the notations in \cite{Dykema:DukeMathJ93}), where we may and do assume that $\alpha_1\geq\alpha_2\geq\cdots\gneqq0$, $\beta_1\geq\beta_2\geq\cdots\gneqq0$ and $\alpha_1 \geq \beta_1$. The free products of this form were already studied in detail by Dykema \cite{Dykema:DukeMathJ93}. His consequence agrees with the statements of Theorem \ref{T-4.1}. (We should point out that he further proved that the diffuse factor part $M_c$ is always isomorphic to an (interpolated) free group factor in this case.) The case of both $K_i < \infty$ ($i=1,2$) is treated as \cite[Theorem 2.3]{Dykema:DukeMathJ93}. However the same tedious and elementary induction argument as there with the help of Theorem \ref{T-3.7} instead of e.g.~\cite[Lemma 1.3, Lemma 1.4, Remark 1.5]{Dykema:DukeMathJ93} (which heavily depend on so-called `random matrix machinery') shows the desired assertions. The case where either $K_1$ or $K_2$ is infinite is treated in \cite[Theorem 4.6]{Dykema:DukeMathJ93}, but the argument presented in Step 3 below reduces this case to the previous one, i.e., both $K_i < \infty$.    

\medskip
{\bf Step 2 -- non-commutative but the centers are finite dimensional -- most essential step.} Assume that both $M_1=M_{1,d}$ and $M_2 = M_{2,d}$ have the finite dimensional centers, that is, both $J_1$ and $J_2$ are finite sets. The proof will be done by induction in the number of non-trivial $B(\mathcal{H})$-components so that we have assumed that both $M_1$ and $M_2$ have the finite dimensional centers. Thanks to Step 1, as induction hypothesis we may assume, by switching $M_1$ and $M_2$ if necessary, that $M$ has the following structure: 
\begin{itemize} 
\item $M_2$ has $B(\mathcal{K})$ (possibly with $\mathrm{dim}(\mathcal{K})=\infty$) as a direct summand, i.e., $M_2 = B(\mathcal{K})\oplus Q_2$, 
\item the density operator of $\varphi_2|_{B(\mathcal{K})}$ is diagonalized by a matrix unit system $\{e_{ij}\}_{i,j}$ of $B(\mathcal{K})$, 
\item With letting 
$$
N_2 := \sideset{}{^\oplus}\sum_i \mathbb{C}e_{ii}\oplus Q_2\ (\subset M_2), 
$$
$N := M_1 \vee N_2$ equipped with $\varphi|_N$ is nothing less than the free product $(M_1,\varphi_1)\star(N_2,\varphi_2|_{N_2})$ that satisfies the following: $N$ is decomposed into a direct sum $N = N_d\oplus N_c$ of a multi-matrix algebra $N_d$ whose structure agrees with the statement of Theorem \ref{T-4.1} and a diffuse factor $N_c$ with $((N_c)_{\varphi|_{N_c}})'\cap N_c^\omega = \mathbb{C}$.    
\end{itemize} 
(Remark that the above assumption on $N_c$ does not hold as it is only when $M_1=\mathbb{C}\oplus\mathbb{C}$ and $M_2=B(\mathcal{K})$ with $\mathrm{dim}(\mathcal{K}) = 2$, but the argument below still works in the case too with the help of the structure theorem on two freely independent projections.) There are two possibilities, that is, 
\begin{itemize}
\item[(2-i)] at least one of the diagonals $e_{ii}$'s falls in $N_c$, 
\item[(2-ii)] no diagonal $e_{ii}$ falls in $N_c$, i.e., both $e_{ii}1_{N_c} \neq 0$ and $e_{ii}1_{N_d}\neq0$ hold for every $i$. 
\end{itemize}
Note that only $\mathrm{dim}(\mathcal{K}) < \infty$ is possible in the case (2-ii).   

\medskip
{\bf Case (2-i).} This case has no counterpart in \cite{Dykema:DukeMathJ93}.  

Let us assume that $e_{kk} \in N_c$ for some $k$. We may assume $k=1$. Consider the following sets of words: 
\begin{equation*} 
e_{1i}\,(\underbrace{M_1^\circ\quad\cdots\quad M_1^\circ}_{\text{alternating in $M_1^\circ, M_2^\circ$}})\,e_{j1}, \quad (\text{for all possible $i,j$}). 
\end{equation*}
Let $\mathcal{X}_1$, $\mathcal{X}_2$, $\mathcal{X}_3$ and $\mathcal{X}_4$ be the closed subspaces, in the standard Hilbert space $\mathcal{H} := L^2(M,\varphi)$ via $\Lambda_\varphi$, generated by 
$e_{11}(M_1^\circ\cdots M_1^\circ)e_{11}$, by $e_{1i}(M_1^\circ\cdots M_1^\circ)e_{j1}$ with $i\neq 1 \neq j$, by $e_{1i}(M_1^\circ\cdots M_1^\circ)e_{11}$ with $i\neq 1$ and by $e_{11}(M_1^\circ\cdots M_1^\circ)e_{j1}$ with $j\neq 1$, respectively. It is easy to see, due to $e_{ij} \in M_2^\circ$ for $i\neq j$, that those $\mathcal{X}_1$, $\mathcal{X}_2$, $\mathcal{X}_3$ and $\mathcal{X}_4$ are mutually orthogonal and moreover that 
\begin{equation*} 
\mathcal{H}_0 := \overline{\Lambda_{\varphi}(e_{11}Me_{11})} = \mathbb{C}\Lambda_{\varphi}(e_{11}) \oplus \mathcal{X}_1 \oplus \mathcal{X}_2 \oplus \mathcal{X}_3 \oplus \mathcal{X}_4.
\end{equation*}
(The last assertion immediately follows from the facts that $M_2 + \mathrm{span}\big(\Lambda^\circ(M_1^\circ,M_2^\circ)\setminus M_2^\circ\big)$ forms a dense $*$-subalgebra of $M$ in any von Neumann algebra topology and that $M_2 e_{11}$ and $e_{11}M_2$ are generated as linear subspaces by the $e_{j1}$'s and the $e_{1i}$'s, respectively.) By assumption one can choose $u \in (N_c)_{\varphi|_{N_c}}$ in such a way that $u^* u=uu^* = e_{11}$ and $\varphi(u^n) = \delta_{n0}\varphi(e_{11})$, since $e_{11} \in N_c$ and $\sigma_t^{\varphi_2}(e_{11}) = e_{11}$ for every $t \in \mathbb{R}$. Then we can define a unitary operator $U$ on $\mathcal{H}_0$ by $U\Lambda_\varphi(x) := \Lambda_\varphi(uxu^*)$ for $x \in e_{11}Me_{11}$. Since $N_2 + \mathrm{span}\big(\Lambda^\circ(M_1^\circ,N_2^\circ)\setminus N_2^\circ\big)$ forms a dense $*$-subalgebra of $N$ in any von Neumann algebra topology, every non-trivial power $u^n$ can clearly be approximated, due to Kaplansky's density theorem, by a bounded net of linear combinations of words in 
$$
e_{11}\,(\underbrace{M_1^\circ\quad\cdots\quad M_1^\circ}_{\text{alternating in $M_1^\circ, \underline{N_2^\circ}$}})\,e_{11}. 
$$
Trivially $\big\{U^n\mathcal{X}_2\big\}_{n\in\mathbb{Z}}$ is a family of mutually orthogonal subspaces of $\mathcal{H}_0$. Since $N_2^\circ e_{1i} = \mathbb{C}e_{1i} = e_{1i}N_2^\circ$ and $N_2^\circ e_{j1} = \mathbb{C}e_{j1} = e_{j1}N_2^\circ$ (both sit in $M_2^\circ$), we can prove that both $\big\{U^n\mathcal{X}_3\big\}_{n\in\mathbb{Z}}$ and $\big\{U^n\mathcal{X}_4\big\}_{n\in\mathbb{Z}}$ also become families of mutually orthogonal subspaces of $\mathcal{H}_0$. The essential point of showing the mutual orthogonality of $\big\{U^n\mathcal{X}_4\big\}_{n\in\mathbb{Z}}$ is as follows. (Confirming that of $\big\{U^n\mathcal{X}_3\big\}_{n\in\mathbb{Z}}$ is easier than this case.) The problem is reduced to showing that any word in 
$$
e_{1j}\underbrace{M_1^\circ\quad\cdots\quad M_1^\circ}_{\text{alternating in $M_1^\circ,M_2^\circ$}}
e_{11}\underbrace{M_1^\circ\quad\cdots\quad M_1^\circ}_{\text{alternating in $M_1^\circ,N_2^\circ$}}
e_{11}\underbrace{M_1^\circ\quad\cdots\quad M_1^\circ}_{\text{alternating in $M_1^\circ,M_2^\circ$}}
e_{j1}\underbrace{M_1^\circ\quad\cdots\quad M_1^\circ}_{\text{alternating in $M_1^\circ,N_2^\circ$}}
e_{11}
$$
is in the kernel of $\varphi$. By approximating each letter by analytic elements we may assume that each letter of the word in question is analytic. Hence by using \cite[Exercise VIII.2(2)]{Takesaki:Book2} again and again we can transform the question to showing the same one for 
$$
e_{11}\underbrace{M_1^\circ\quad\cdots\quad M_1^\circ}_{\text{alternating in $M_1^\circ,\underline{N_2^\circ}$}}
e_{11}\underbrace{M_1^\circ\quad\cdots\quad M_1^\circ}_{\text{alternating in $M_1^\circ,M_2^\circ$}}
e_{j1}\underbrace{M_1^\circ\quad\cdots\quad M_1^\circ}_{\text{alternating in $M_1^\circ,N_2^\circ$}}
e_{1j}\underbrace{M_1^\circ\quad\cdots\quad M_1^\circ}_{\text{alternating in $M_1^\circ,M_2^\circ$}}. 
$$   
({\it n.b.}~$\varphi_i(\sigma_z^{\varphi_i}(x)) = \varphi_i(x)$, $z \in \mathbb{C}$, holds for every analytic $x$, $i=1,2$.) This question can easily be settled by using $N_2^\circ e_{j1} = \mathbb{C}e_{j1} \subseteq M_2^\circ$. 

Choose arbitrary $x \in \{u\}'\cap(e_{11}Me_{11})^\omega = \{u\}'\cap e_{11}M^\omega e_{11}$ with representative $(x(m))_m$. Denote by $P_{\mathcal{X}_k}$ the projection from $\mathcal{H}_0$ onto $\mathcal{X}_k$ for $k = 1,2,3,4$. Then we can prove, in the exactly same way as in the proof of Proposition \ref{P-3.5}, that for a given $\gamma>0$ there is a neighborhood at $\omega$ on which 
$$
\left\Vert P_{\mathcal{X}_k}\Lambda_\varphi(x(m))\right\Vert_\varphi < \gamma, \quad k=2,3,4.   
$$
By the assumption on $N_c$ one can choose an invertible element $y_{\ell}^\circ$ of $e_{\ell \ell}Ne_{\ell \ell}$ with $\ell \neq 1$ in such a way that $\varphi(y_\ell^\circ) = 0$ and $\sigma_t^\varphi(y_\ell^\circ) = y_\ell^\circ$ for every $t \in \mathbb{R}$. (See the proof of Theorem \ref{T-3.7}.) Set $y^\circ := e_{1\ell}y_\ell^\circ e_{\ell 1}$. Clearly $\sigma_t^\varphi(y^\circ)=y^\circ$ holds for every $t \in \mathbb{R}$ (since the $e_{ij}$'s diagonalize the density operator of $\varphi_2|_{B(\mathcal{K})}$) and $y^\circ$ can be approximated, due to Kaplansky's density theorem, by a bounded net consisting of linear combinations of words in $e_{1\ell}(M_1^\circ\cdots M_1^\circ)e_{\ell 1}$. Since 
\begin{align*} 
y^\circ\Lambda_\varphi(e_{11}(M_1^\circ\cdots M_1^\circ)e_{11}) &\subseteq \Lambda_\varphi(e_{1\ell}(M_1^\circ\cdots M_1^\circ)e_{11}), \\ 
J(y^\circ)^*J\Lambda_\varphi(e_{11}(M_1^\circ\cdots M_1^\circ)e_{11}) &= 
\Lambda_\varphi(e_{11}(M_1^\circ\cdots M_1^\circ)e_{11}y^\circ) \subseteq 
\Lambda_\varphi(e_{11}(M_1^\circ\cdots M_1^\circ)e_{\ell 1})   
\end{align*} 
we see, as in the proof of Proposition \ref{P-3.5}, that    
\begin{align*} 
\Lambda_{\varphi^\omega}(y^\circ(x-(1/\varphi_2(e_{11}))\varphi^\omega(x)e_{11})) 
&= 
\Big[\big(y^\circ P_{\mathcal{X}_1}\Lambda_\varphi(x(m))\big)_m\Big], \\
\Lambda_{\varphi^\omega}((x-(1/\varphi_2(e_{11}))\varphi^\omega(x)e_{11})y^\circ) 
&= 
\Big[\big(J(y^\circ)^*J P_{\mathcal{X}_1}\Lambda_\varphi(x(m))\big)_m\Big]
\end{align*} 
are orthogonal to each other in the ultraproduct $\mathcal{H}_0^\omega$ ($\subset \mathcal{H}^\omega$), where $J$ is the modular conjugation of $M \curvearrowright \mathcal{H}$. This immediately implies that $\Vert y^\circ(x-(1/\varphi_2(e_{11}))\varphi^\omega(x)e_{11})\Vert_{\varphi^\omega} \leq \Vert [x,y^\circ] \Vert_{\varphi^\omega}$. Therefore $ (e_{11} M_\varphi e_{11})' \cap (e_{11}M^\omega e_{11}) \subseteq \{u,y^\circ\}'\cap (e_{11}Me_{11})^\omega = \mathbb{C}e_{11}$ since $y^\circ$ is invertible in $e_{11}Me_{11}$. Note that every $e_{ii}Me_{ii}\supseteq e_{ii}M_\varphi e_{ii}$ is conjugate to $e_{11}Me_{11} \supseteq e_{11}M_\varphi e_{11}$ via $\mathrm{Ad}e_{1i}$, and in particular, every $e_{ii}M_\varphi e_{ii}$ is a factor. Hence one can see that $c_{M_\varphi}(e_{11}) = (\sum_i 1_{N_d}e_{ii}) + 1_{N_c}$, since every $e_{ii}$ has a non-trivial part in $(N_c)_{\varphi|_{N_c}}$ by the induction hypothesis here. Consequently $M = M_d\oplus M_c$ so that $M_d = N_d(1_{N_d}-\sum_i 1_{N_d}e_{ii})$ and $\big((M_c)_{\varphi|_{M_c}}\big)'\cap M_c^\omega = \mathbb{C}$. In particular, $B(\mathcal{K})$ has no contribution to the multi-matrix part $M_d$, and this agrees with the condition \eqref{Eq-4.1}.  

\medskip
{\bf Case (2-ii).} We borrow some ideas from the proof of \cite[Proposition 3.2]{Dykema:DukeMathJ93}, and then apply what we have provided in \S3 straightforwardly. 

Since $N_2 + \mathrm{span}\big(\Lambda^\circ(M_1^\circ,N_2^\circ)\setminus N_2^\circ\big)$ forms a dense $*$-subalgebra in $N$ in any von Neumann algebra topology and also since every $e_{ii}$ is minimal and central in $N_2$, the  von Neumann subalgebra $e_{1i} Ne_{i1}$ of $e_{11}Me_{11}$ $(1 \leqq i \leqq n$) is the closure of the linear span of $e_{11}$ and 
\begin{equation*} 
e_{1i}\,(\underbrace{M_1^\circ\quad\cdots\quad M_1^\circ}_{\text{alternating in $M_1^\circ, N_2^\circ$}})\,e_{i1} 
\end{equation*}
in any von Neumann algebra topology, and thus the kernel of $\varphi|_{e_{1i}Ne_{i1}}$ is the closure of the linear span of $e_{1i}(M_1^\circ\cdots M_1^\circ)e_{i1}$'s in the same topology. It follows that the  $e_{1i}Ne_{i1}$'s are free in $\big(e_{11}Me_{11}, (1/\varphi(e_{11}))\,\varphi|_{e_{11}Me_{11}}\big)$. By the assumption here, there is a minimal and central projection $p \in M_1$ such that the multi-matrix part $N_d$ has the component 
$$
\overset{p\wedge e_{11}}{\underset{\varphi_1(p)+\varphi_2(e_{11})-1}{\mathbb{C}}}\oplus\cdots\oplus\overset{p\wedge e_{nn}}{\underset{\varphi_1(p)+\varphi_2(e_{nn})-1}{\mathbb{C}}} \quad \text{with all $\varphi_1(p)+\varphi_2(e_{ii})-1 \gneqq 0$} 
$$
with $n=\mathrm{dim}(\mathcal{K}) < +\infty$ and moreover that every $e'_{ii} := e_{ii}-p\wedge e_{ii}$ falls in $(N_c)_{\varphi|_{N_c}}$. Notice that every $\varphi(e'_{ii})$ is $1-\varphi_1(p)$ since \eqref{Eq-4.4} holds, i.e., $\varphi(p\wedge e_{ii}) = \varphi_1(p) + \varphi_2(e_{ii}) -1$, by the assumption on $N_d$, and hence all $e'_{ii}$ are (Murray--von Neumann) equivalent to each other in $(N_c)_{\varphi|_{N_c}}$ since $(N_c)_{\varphi|_{N_c}}$ is a factor. (Here and in the next line the structure theorem on two freely independent projections is necessary when $M_1 = \mathbb{C}\oplus\mathbb{C}$ and $M_2 = B(\mathcal{K})$ with $\mathrm{dim}(\mathcal{K}) = 2$.) Therefore we can choose partial isometries $y_{i1}$ ($2\leq i \leq n$) from $(N_c)_{\varphi|_{N_c}}$ in such a way that $y_{i1}^*y_{i1} = e'_{11}$ and $y_{i1}y_{i1}^* = e'_{ii}$. The von Neumann subalgebra $P$ generated by $e_{11}Ne_{11} = \mathbb{C}(p\wedge e_{11})+ e'_{11}N_c e'_{11}$ and the $\mathbb{C}e_{1i}(p\wedge e_{ii})e_{i1} + \mathbb{C}e_{1i}e'_{ii}e_{i1}$'s ($2 \leq i \leq n$) in $e_{11}Me_{11}$ is nothing but the $n$-fold free product von Neumann algebra of  
$$
\overset{p\wedge e_{11}}{\underset{1-\frac{1-\varphi_1(p)}{\varphi_2(e_{11})}}{\mathbb{C}}}\oplus\overset{q_1 := e'_{11}}{\underset{\frac{1}{\varphi_2(e_{11})}\varphi|_{e'_{11}N_c e'_{11}}}{e'_{11}N_c e'_{11}}} \quad \text{and} \quad \overset{e_{1i}(p\wedge e_{ii})e_{i1}}{\underset{1-\frac{1-\varphi_1(p)}{\varphi_2(e_{ii})}}{\mathbb{C}}}\oplus\,\overset{q_i := e_{1i}e'_{ii} e_{i1}}{\underset{\frac{1-\varphi_1(p)}{\varphi_2(e_{ii})}}{\mathbb{C}}} \quad (2 \leqq i \leqq n). 
$$
In what follows we may and do assume that $\varphi_2(e_{11})$ is the smallest among the $\varphi_2(e_{ii})$'s (and thus $1-\frac{1-\varphi_1(p)}{\varphi_2(e_{11})} \leq 1-\frac{1-\varphi_1(p)}{\varphi_2(e_{ii})}$). The inductive use of Theorem \ref{T-3.7} and Lemma \ref{L-2.2} together with the structure theorem on two freely independent projections enables us to show that $P = \overset{r}{\mathbb{C}}\oplus P_c$, where
\begin{equation*} 
r := \bigwedge_{i=1}^n e_{1i}(p\wedge e_{ii})e_{i1} \quad 
\varphi(r) = \varphi_2(e_{11})\Big(\max\Big\{1-(1-\varphi_1(p))\sum_{i=1}^n\frac{1}{\varphi_2(e_{ii})},0\Big\}\Big)
\end{equation*} 
(possibly with $r=0$), and also $((P_c)_{\varphi|_{P_c}})'\cap P_c^\omega = \mathbb{C}$ holds since $(e'_{11}N_c e'_{11})_{\varphi|_{e'_{11}N_c e'_{11}}}$ is diffuse. We also have $1_{P_c} = 1_P-r = \bigvee_{i=1}^n q_i$, since $q_i = 1_P - e_{1i}(p\wedge e_{ii})e_{i1}$. Let us prove $((q_1 Mq_1)_{\varphi|_{q_1 Mq_1}})'\cap(q_1 M^\omega q_1) = \mathbb{C}q_1$. 

Consider the von Neumann subalgebra $\{q_1,q_i\}''$ in $P$ ($2 \leq i \leq n$). The structure theorem on two freely independent projections enables us to choose a partial isometry $z_i \in \{q_1,q_i\}''$ so that $z_i^* z_i = q_i$ and $z_i z_i^* \leq q_1$, since $\varphi_2(e_{11}) \leq \varphi_2(e_{ii})$ (and thus $\varphi(q_1) \geq \varphi(q_i)$, $2\leq i \leq n$). Then $w_i := z_i e_{1i} y_{i1}$ is an isometry in $q_1 M q_1$ with $w_i w_i^* = z_i z_i^* \leq q_1$. Note that $e_{11}M e_{11}$ is generated by $q_1 N q_1 = q_1 N_c q_1$ and the $e_{1i}y_{i1}$'s together with $e_{11}$, and hence $q_1 M q_1$ is generated by $q_1 N q_1 = q_1 N_c q_1$, $q_1\{q_1,q_2,\dots,q_n\}''q_1$ and the $w_i$'s. (To see this, insert $q_i = z_i^*z_i$ before each $w_i$ and after each $w_i^*$ in any possible word in $q_1 N q_1$ and the $w_i$, $w_i^*$'s, and then regroup the resulting word.) We write $Q := \{q_1,q_2,\dots,q_n\}''$ in $P$ and also $q'_i := z_i z_i^*$ ($\leq q_1$), $2 \leq i \leq n$, for simplicity.       

\begin{claim} {\rm ({\it c.f.}~\cite[Claim 3.2b and Claim 3.2c]{Dykema:DukeMathJ93})} The `restricted' traveling words in  
$(q_1 Q q_1)^\circ = \mathrm{Ker}(\varphi|_{q_1 Q q_1})$, $\{w_i^{\ell}, (w_i^*)^{\ell}\,|\, \ell \in \mathbb{N}\}$, $2\leqq i \leqq n$, form a total subset of the kernel of the restriction of $\varphi$ to $q_1 Q q_1 \vee \{w_i\,|\,2\leq i\leq n\}''$ in any von Neumann algebra topology. Here a traveling word $x_1 x_2 \cdots x_\ell$ is said to be `restricted' if $x_k = q'_i x_k q'_i$ holds {\rm(}i.e., $x_k$ must fall in $(q'_i Q q'_i)^\circ = \mathrm{Ker}(\varphi|_{q'_i Q q'_i})${\rm )} when $x_{k-1} = (w_i^*)^{\ell_1}$, $x_k \in (q_1 Q q_1)^\circ$ and $x_{k+1} = (w_i)^{\ell_2}$ with some $\ell_1,\ell_2 \in \mathbb{N}$. Moreover any `restricted' traveling word in $(q_1 Q q_1)^\circ$, $\{w_i^{\ell}, (w_i^*)^{\ell}\,|\, \ell \in \mathbb{N}\}$'s and $(q_1 N q_1)^\circ = \mathrm{Ker}(\varphi|_{q_1 N q_1})$ {\rm(}in the above sense{\rm)} is in the kernel of $\varphi$.           
\end{claim} 
\begin{proof} Notice that $w_i^* w_i = q_1$ and $w_i w_i^* = q'_i \in q_1\{q_1,q_2,\dots,q_n\}''q_1$, and thus it is not difficult, by using the decompositions 
\begin{align*} 
x &\mapsto \frac{\varphi(x)}{\varphi(q_1)}q_1 + \big(x-\frac{\varphi(x)}{\varphi(q_1)}q_1\big) \in \mathbb{C}q_1 + \mathrm{Ker}(\varphi) \quad \text{(if $x \in q_1 N q_1 \cup q_1 Q q_1$)}, \\
x &\mapsto \frac{\varphi(x)}{\varphi(q'_i)}q'_i + \big(x-\frac{\varphi(x)}{\varphi(q'_i)}q'_i\big) \in \mathbb{C}q'_i + (q'_i Q q'_i)^\circ \quad \text{(if $x \in q'_i Q q'_i$, $2\leq i \leq n$)}
\end{align*} 
again and again, to see that $q_1$ and all the possible `restricted' traveling words in $(q_1 Q q_1)^\circ$, $\{w_i^{\ell}, (w_i^*)^{\ell}\,|\, \ell \in \mathbb{N}\}$'s form a total subset of $q_1 Q q_1 \vee \{w_i\,|\,2\leq i\leq n\}''$ in any von Neumann algebra topology. Hence it suffices to prove that $\varphi(x) = 0$ for any `restricted' traveling word $x$ in $(q_1 Q q_1)^\circ$, $\{w_i^{\ell}, (w_i^*)^{\ell}\,|\, \ell \in \mathbb{N}\}$'s and $(q_1 N q_1)^\circ$. In the rest of the proof we need to give heed to the following simple fact: $z_i^* (q'_i Q q'_i)^\circ z_i \subseteq (q_i Q q_i)^\circ$. This is due to $\sigma_t^\varphi(z_i) = z_i$ for every $t \in \mathbb{R}$. In fact this fact is the reason why usual traveling words are not suitable and `restricted' ones are necessary here. Although the discussion below is almost the same as in \cite[Claim 3.2c]{Dykema:DukeMathJ93}, we do give a sketch for the reader's convenience. 

Regrouping a given `restricted' traveling word we can make it an alternating word in 
$$
\Omega_1 = (q_1 N q_1)^\circ \cup \bigcup_{2\leq i \leq n}\big(q_1 N y_{i1}^* \cup y_{i1}Nq_1\cup y_{i1}(q_1 N q_1)^\circ y_{i1}^*\big) \cup \bigcup_{2\leq i \neq j \leq n} y_{i1} N y_{j1}^*, 
$$ $$
\Omega_2 = (q_1 Q q_1)^\circ \cup \bigcup_{2\leq i \leq n}\big(e_{i1}z_i^* Q q_1 \cup q_1 Q z_i e_{1i} \cup e_{i1}z_i^*(q'_i Q q'_i)^\circ z_i e_{1i}\big) \cup \bigcup_{2\leq i \neq j \leq n} e_{i1}z_i^* Q z_j e_{1j}
$$
with some constraints due to the fact that $w_i = z_i e_{1i}y_{i1}$, $2 \leq i \leq n$, and their adjoints appear, as blocks,  in the given `restricted' traveling word. Then we firstly approximate each letter from $\Omega_2$ by linear combinations of words in $c_i := e_{1i}\tilde{c}_i e_{i1}$ with $\tilde{c}_i := e'_{ii} - \frac{\varphi(e'_{ii})}{\varphi_2(e_{ii})}e_{ii}$, $2\leq i \leq n$, and $e_{ij}$'s. For example any element in $e_{i1}z^*_i(q'_i Q q'_i)^\circ z_i e_{1i} \subseteq e_{i1}(q_i Q q_i)^\circ e_{1i}$ can be approximated, due to Kaplansky's density theorem, by a bounded net consisting of linear combinations of words of the form: 
$$
e_{i1} q_i \underbrace{c_i c_{i_2}\cdots c_{i_{n-1}} c_i}_{\text{non-trivial, traveling}} q_i e_{1i} = e'_{ii}\tilde{c}_i e_{ii_2}\tilde{c}_{i_2}e_{i_2 i_3}\cdots e_{i_{n-1}i}\tilde{c}_i e'_{ii}.    
$$
(Note here that $q_i c_i$ is a scalar multiple of $q_i$.) Hence for any $x \in (q'_i Q q'_i)^\circ$ the element $y_{i1}^* e_{i1} z_i^* x z_i e_{1i}y_{i1}$ is approximated by linear combinations of 
$$
y_{i1}^* e'_{ii}\tilde{c}_i e_{ii_2}\tilde{c}_{i_2}e_{i_2 i_3}\cdots e_{i_{n-1}i}\tilde{c}_i e'_{ii} y_{i1} = 
(y_{i1}^* \tilde{c}_i) e_{ii_2} \tilde{c}_{i_2} e_{i_2 i_3} \cdots e_{i_{n-1}i}(\tilde{c}_i y_{i1}). 
$$
Since $y_{i1}^* \tilde{c}_i \in e_{11} N e_{ii}$ and $\tilde{c}_i y_{i1} \in e_{ii} N e_{11}$, we finally approximate the right-hand side above by linear combinations of words in 
$$
e_{11}\underbrace{M_1^\circ\cdots M_1^\circ}_{\text{alternating in $M_1^\circ, N_2^\circ$}}e_{ii_2}\underbrace{M_1^\circ\cdots M_1^\circ}_{\text{alternating in $M_1^\circ, N_2^\circ$}}e_{i_2 i_3}\quad \cdots\quad e_{i_{n-1}i}\underbrace{M_1^\circ\cdots M_1^\circ}_{\text{alternating in $M_1^\circ, N_2^\circ$}}e_{11}               
$$
which can be written as a linear combination of alternating words in $M_1^\circ, M^\circ_2$ since $i \neq i_2 \neq i_3 \neq \cdots \neq i_{n-1} \neq i$. (Remark here that $N_2+\mathrm{span}\big(\Lambda^\circ(M_1^\circ,N_2^\circ)\setminus N_2^\circ\big)$ forms a dense $*$-subalgebra of $N$ in any von Neumann algebra topology.) In this way any `restricted' traveling word can be approximated, due to Kaplansky's density theorem, by a bounded net consisting of linear combinations of alternating words in $M_1^\circ, M_2^\circ$. Hence we are done. 
\end{proof} 

It follows from the above claim that $q_1 N q_1$ and $q_1 Q q_1 \vee \{w_i\,|\,2\leq i \leq n\}''$ are free in $(q_1 M q_1,(1/\varphi(q_1))\varphi|_{q_1 M q_1})$. Note that $(q_1 N q_1)_{\varphi|_{q_1 N q_1}} = q_1 (N_c)_{\varphi_{N_c}} q_1$ is clearly diffuse, and also $q_1 Q q_1$ is non-trivial and sits in $(q_1 M q_1)_{\varphi|_{q_1 M q_1}}$. Therefore (the latter assertion of) Theorem \ref{T-3.7} shows $(q_1 M_\varphi q_1)'\cap(q_1 M^\omega q_1) = ((q_1 M q_1)_{\varphi|_{q_1 M q_1}})' \cap (q_1 M q_1)^\omega = \mathbb{C}q_1$. Since $c_{P_{\varphi|_P}}(q_1) = e_{11}-r$ ({\it n.b.}~$e_{11}M_\varphi e_{11}$ contains $P_{\varphi|_P}$) and since $r = \sum_{i=1}^n e_{1i}(p\wedge e_{ii})e_{i1}$ is minimal and central in $e_{11}Me_{11}$ ({\it n.b.}~$e_{11}Me_{11}$ is generated by $P$ and the $e_{1i}y_{i1}$'s), one has 
\begin{align}
&e_{11}M e_{11} = 
\overset{r}{\mathbb{C}}\oplus(e_{11}-r)M(e_{11}-r), \label{Eq-4.8}\\ 
&((e_{11}-r)M_\varphi(e_{11}-r))'\cap((e_{11}-r)M^\omega(e_{11}-r)) = \mathbb{C}(e_{11}-r). \label{Eq-4.9}
\end{align}
Then, by \eqref{Eq-4.8} we get
\begin{align*} 
1_{B(\mathcal{K})}M1_{B(\mathcal{K})} = 
\mathrm{span}\{e_{i1}re_{1j}\,|\,1\leq i,j\leq n\}\oplus\big(\sum_{i=1}^n e_{i1}(e_{11}-r)e_{1i}\big)M\big(\sum_{i=1}^n e_{i1}(e_{11}-r)e_{1i}\big)\end{align*}
since $\{e_{ij}\}_{i,j}$ is a unital matrix unit system inside the left-hand side. Therefore this description of $1_{B(\mathcal{K})}M1_{B(\mathcal{K})}$, the relative commutant property \eqref{Eq-4.9} and $e_{11}-r \geq q_1 = e'_{11} \in (N_c)_{\varphi|_{N_c}}$ altogether show $c_M(e_{11}-r) = \big(\sum_{i=1}^n e_{i1}(e_{11}-r)e_{1i}\big)\vee 1_{N_c} = \big(\sum_{i=1}^n (e_{ii}-e_{i1}re_{1i})\big)\vee 1_{N_c} = \big(\sum_{i=1}^n (p\wedge e_{ii} - e_{i1}re_{1i})\big)+1_{N_c}$. Note that $(e_{11}-r)M_\varphi(e_{11}-r)$ is conjugate, via $\mathrm{Ad}e_{i1}$, to $(e_{ii}-e_{i1}re_{1i})M_\varphi(e_{ii}-e_{i1}re_{1i})$, and thus every $(e_{ii}-e_{i1}re_{1i})M_\varphi(e_{ii}-e_{i1}re_{1i})$ is a factor, where we remark that $e_{ii}-e_{i1}re_{1i}$ falls in $M_\varphi$. Hence we get $c_{M_\varphi}(e_{ii}-e_{i1}re_{1i}) = c_{M_\varphi}(e'_{ii}) = c_{M_\varphi}(e'_{11}) = c_{M_\varphi}(e_{11}-r)$ for every $i$, since all $e'_{ii}$'s are Murray--von Neumann equivalent in $N_\varphi$ ($\subseteq M_\varphi$). Consequently we have $\big(\sum_{i=1}^n (p\wedge e_{ii} - e_{i1}re_{1i})\big)+1_{N_c} = c_M(e_{11}-r) \geq c_{M_\varphi}(e_{11}-r) = c_{M_\varphi}(e_{ii}-e_{i1}re_{1i}) \geq (e_{ii}-e_{i1}re_{1i})\vee 1_{N_c}$ for every $i$ so that $c_{M_\varphi}(e_{11}-r) = \big(\sum_{i=1}^n (p\wedge e_{ii} - e_{i1}re_{1i})\big)+1_{N_c}$. Therefore we conclude $M_d = N_d(1_{N_d}-\sum_{i=1}^n p\wedge e_{ii})\oplus\mathrm{span}\big\{ e_{i1}re_{1j}\,|\, 1\leq i, j \leq n\}$ and $((M_c)_{\varphi|_{M_c}})'\cap M_c^\omega = \mathbb{C}$. The description of $\mathrm{span}\big\{ e_{i1}re_{1j}\,|\, 1\leq i, j \leq n\} \cong B(\mathcal{K})$ in $M_d$ agrees with \eqref{Eq-4.3}. We also have 
\begin{align*} 
\varphi(e_{i1}re_{1j}) 
= \delta_{ij}\frac{\varphi_2(e_{ii})}{\varphi_2(e_{11})}\varphi(r) 
= \delta_{ij}\varphi_2(e_{ii})\Big(1-(1-\varphi_1(p))\sum_{k=1}^n\frac{1}{\varphi_2(e_{kk})}\Big)
\end{align*} 
(if $r\neq0$), which agrees with \eqref{Eq-4.4}. Therefore we complete the discussion of (2-ii). In this case $B(\mathcal{K})$ has contribution to the multi-matrix part $M_d$ as long as $r \neq 0$ or equivalently $\sum_{i=1}^n \frac{1}{\varphi_2(e_{ii})} \lneqq \frac{1}{1-\varphi_1(p)}$, which agrees with the condition \eqref{Eq-4.1}. 
  
\medskip
{\bf Step 3 -- general case.} It remains to be dealt with the case that either $M_1=M_{1,d}$ or $M_2 = M_{2,d}$ has the infinite dimensional center, that is, either $J_1$ or $J_2$ is an infinite set. Firstly suppose that $J_1$ is finite but $J_2$ is infinite. Then, set:  
$$
N_2 := \Big[\sideset{}{^\oplus}\sum_{j \in J'_2} B(\mathcal{H}_{2j})\Big]\oplus\mathbb{C}1_{J'_2{}^c}
$$
for a finite subset $J'_2 \Subset J_2$, where $1_{J'_2{}^c}$ denotes the unit of $N_{2,J'_2{}^c} := \sideset{}{^\oplus_{j \in J'_2{}^c}}\sum B(\mathcal{H}_{2j})$. Note that $N := M_1\vee N_2$ is a free product von Neumann algebra that is treated in Step 2 so that $N = N_d\oplus N_c$, where $N_d$ is a multi-matrix algebra whose structure is determined by the algorithm in Theorem \ref{T-4.1} and $((N_c)_{\varphi|_{N_c}})'\cap N_c^\omega = \mathbb{C}$. One can choose $J'_2$ such that $\varphi_2(1_{J'_2{}^c})$ is so small, and hence $1_{J'_2{}^c} \in N_c$. By Lemma \ref{L-2.2} we see that $1_{J'_2{}^c} M 1_{J'_2{}^c}$ is the free product von Neumann algebra of $N_{2,J'_2{}^c}$ with $(1/\varphi_2(1_{J'_2{}^c}))\,\varphi_2|_{N_{2,J'_2{}^c}}$ and $1_{J'_2{}^c} N 1_{J'_2{}^c} = 1_{J'_2{}^c} N_c 1_{J'_2{}^c}$ with $(1/\varphi_2(1_{J'_2{}^c}))\varphi|_{1_{J'_2{}^c} N 1_{J'_2{}^c}}$, and moreover $c_M(1_{J'_2{}^c}) = c_N(1_{J'_2{}^c}) = 1_{N_c}$. Thus Theorem \ref{T-3.7} shows $((1_{J'_2{}^c} N 1_{J'_2{}^c})_{\varphi|_{1_{J'_2{}^c} N 1_{J'_2{}^c}}})'\cap(1_{J'_2{}^c} M 1_{J'_2{}^c})^\omega = \mathbb{C}$. Since $1_{J'_2{}^c} \in (N_c)_{\varphi|_{N_c}}$ and $((N_c)_{\varphi|_{N_c}})'\cap N_c^\omega = \mathbb{C}$, one has $c_{N_\varphi}(1_{J'_2{}^c}) = 1_{N_c}$ so that $1_{N_c} \geq c_M(1_{J'_2{}^c}) \geq c_{M_\varphi}(1_{J'_2{}^c}) \geq c_{N_\varphi}(1_{J'_2{}^c}) = 1_{N_c}$ implying $c_{M_\varphi}(1_{J'_2{}^c})=1_{N_c}$.  Consequently, $M = M_d\oplus M_c$, $M_d = N_d$ and moreover $((M_c)_{\varphi|_{M_c}})'\cap M_c^\omega = \mathbb{C}$. 

When both $J_1$ and $J_2$ are infinite, the same argument reduces this case to the previous one, i.e., one of $J_i$'s is infinite and the other finite. Hence we are done.      

\medskip
We complete the proof of the case (d). Hence the proof of Theorem \ref{T-4.1} is now finished.        

\section{Concluding Remarks} 

\subsection{On $((M_c)_{\varphi|_{M_c}})'\cap M_c = \mathbb{C}$} Let us further assume that $M_1$ and $M_2$ have separable preduals. It is known, see e.g.~\cite[Proposition 4.2]{Dykema:FieldsInstituteComm97}, that the free product state is almost periodic if so are given faithful normal states. We also show in Theorem \ref{T-4.1} that the diffuse factor part $M_c$ is always a full factor. Therefore it is important, in view of Sd-invariant of Connes \cite{Connes:JFA74}, to see when $((M_c)_{\varphi|_{M_c}})'\cap M_c = \mathbb{C}$ holds. If this was true, then the Sd-invariant would coincide with the point spectra of the modular operator $\Delta_\varphi$, which is computed as the (multiplicative) group algebraically generated by the point spectra of the modular operators $\Delta_{\varphi_i}$'s. This is indeed the case; namely we can confirm that $((M_c)_{\varphi|_{M_c}})'\cap M_c^\omega = \mathbb{C}$ holds when both $\varphi_1$ and $\varphi_2$ are almost periodic, though the general situation is much complicated (indeed it does not hold in general !). On the other hand we can show that the free product state is `special' in some sense. 
Those will be given in a separate paper \cite{Ueda:progress}. 

\subsection{$N_\omega = \mathbb{C} \Longrightarrow N'\cap N^\omega = \mathbb{C}$ ?} Our main theorem (Theorem \ref{T-4.1}) says that the diffuse factor part $M_c$ satisfies $(M_c)_\omega = \mathbb{C}$, and furthermore $(M_c)'\cap M_c^\omega = \mathbb{C}$. Thus the next question seems interesting. {\it Does there exist an example of properly infinite factor $N$ with $N_\omega=\mathbb{C}$ but $N'\cap N^\omega \neq \mathbb{C}$ ?}  

\subsection{Lack of Cartan subalgebras} As remarked after Corollary \ref{C-4.3} the diffuse factor part $M_c$ has no Cartan subalgebra when $M$ (or $M_c$) has the weak$^*$ completely bounded approximation property. It is quite interesting whether or not the same phenomenon occurs in general. Remark here that the phenomenon holds for any `tracial' free product of $R^\omega$-embeddable von Neumann algebras due to free entropy technologies \cite{Voiculescu:GAFA96},\cite{Jung:TAMS03},\cite{Shlyakhtenko:TAMS05}. However the question is non-trivial even for `tracial' free products without assuming the $R^\omega$-embeddability. 

\subsection{Questions related to free Araki--Woods factors} It should be a next important question whether or not the diffuse factor part $M_c$ is isomorphic to a free Araki--Woods factor introduced by Shlyakhtenko \cite{Shlyakhtenko:PacificJMath97} when given $M_1$ and $M_2$ are hyperfinite.  The reason is that free Araki--Woods factors are expected to be natural models of `free product type' von Neumann algebras. In the direction Houdayer \cite{Houdayer:IMRN07} could successfully identify some free product factors of hyperfinite von Neumann algebras (including all free products of two copies of $M_2(\mathbb{C})$) with some of free Araki--Woods factors in state--preserving way. 
         
\section*{Acknowledgment} We thank the referee for his or her careful reading and comments.

\end{document}